\documentclass[reqno,10pt]{amsart}
\usepackage{amssymb,amsmath,amsthm,}
\usepackage{a4wide}
\usepackage{amscd}
\usepackage{amsfonts}
\usepackage{amssymb}
\usepackage{latexsym}
\usepackage{color}
\usepackage{esint}
\usepackage{graphicx}
\usepackage{caption}
\usepackage{subcaption}
\usepackage{amsthm}
 \usepackage{verbatim}

\usepackage[makeroom]{cancel}
\usepackage{mathtools}
\DeclarePairedDelimiter\abs{\lvert}{\rvert}

\let\hide\iffalse

\newtheorem{theorem}{Theorem}[section]

\newtheorem{lemma}[theorem]{Lemma}
\newtheorem{proposition}[theorem]{Proposition}
\newtheorem{remark}[theorem]{Remark}

\let\p=\partial
\let\t=\vartheta
\let\O=\Omega

\let\o=\omega
\let\g=\gamma

\let\t=\theta
\let\b=\beta

\newcommand{\R}{\mathbb{R}}
\newcommand{\T}{\mathbb{T}}

\renewcommand{\S}{\mathbb{S}}

\newcommand{\be}{\begin{equation}}
\newcommand{\bm}{\begin{multline}}
\newcommand{\ee}{\end{equation}}

\newcommand{\Bes}{\begin{eqnarray*}}
	\newcommand{\Ees}{\end{eqnarray*}}
\newcommand{\Be}{\begin{equation}}
\newcommand{\Ee}{\end{equation}}

\def\p{\partial}

\def\O{\Omega}
\def\R{\mathbb{R}}

\def\B{\begin{equation}}
\def\E{\end{equation}}
\def\BN{\begin{eqnarray*}}
\def\EN{\end{eqnarray*}}

\newcommand{\om}{\omega}

\numberwithin{equation}{section}

\usepackage{tikz-cd}
\usepackage{ragged2e}
\usepackage{etoolbox}
\usepackage{lipsum}

\usepackage{color}

\begin{document}

\title[The large amplitude solution of the Boltzmann equation with soft potential]{The large amplitude solution of the Boltzmann equation with soft potential}

\author[G. Ko]{Gyounghun Ko}
\address[GK]{Department of Mathematics, Pohang University of Science and Technology, South Korea}
\email{gyounghun347@postech.ac.kr}

\author[D. Lee]{Donghyun Lee}
\address[DL]{Department of Mathematics, Pohang University of Science and Technology, South Korea}
\email{donglee@postech.ac.kr}

\author[K.Park]{Kwanghyuk Park}
\address[KP]{Department of Mathematics, Pohang University of Science and Technology, South Korea}
\email{pkh0219@postech.ac.kr}

\begin{abstract}
	In this paper, we deal with the (angular cut-off) Boltzmann equation with soft potential ($-3 < \gamma < 0$). In particular, we construct a unique global solution in $L^\infty_{x,v}$ which converges to global equilibrium asymptotically provided that initial data has a large amplitude but with sufficienlty small relative entropy. Because frequency multiplier is not uniformly positive anymore, unlike hard potential case, time-involved velocity weight will be used to derive sub-exponential decay of the solution. Motivated by recent development of $L^{2} \mbox{-} L^{\infty}$ approach also, we introduce some modified estimates of quadratic nonlinear terms. Linearized collision kernel will be treated in a subtle manner to control singularity of soft potential kernel.
\end{abstract}

\date{\today}
\keywords{Boltzmann equation, soft potential, large-amplitude initial data, Gronwall argument}
\maketitle

\thispagestyle{empty}
\setcounter{tocdepth}{2}
\tableofcontents

\section{Introduction}

The Boltzmann equation is one of the fundamental mathematical models for collisional rarefied gas theory. The equation describes behavior of the system via distribution function $F(t,x,v)$ where $(t,x,v)\in [0,\infty) \times \O \times \R^3$ where $\O$ is spatial domain and $\R^3$ is a domain for velocity variable. In this paper we choose periodic domain $\O = \T^3$. The Boltzmann equation takes the following form:
\begin{equation} \label{def.be}
\p_{t}F + v\cdot\nabla_{x} F = Q(F,F),
\end{equation}
where the bilinear Boltzmann collision operator acting only on velocity variables $v$ is given by
\begin{align} 
Q(G,F)(v)&=\int_{\mathbb{R}^3}\int_{\mathbb{S}^2} B(v-u,\om)G(u')F(v') \,d\omega du
 - \int_{\mathbb{R}^3}\int_{\mathbb{S}^2} B(v-u,\om)G(u)F(v)\,d\omega du \nonumber \\
&:= Q_+(G,F)-Q_-(G,F). \label{def.Q}
\end{align}
Here the post-collisional velocity $(u',v')$ and the pre-collisional velocity $(u,v)$ satisfy the relation
\begin{equation*}
u'=u+[(v-u)\cdot\omega]\,\omega,\quad
v'=v-[(v-u)\cdot\omega]\,\omega,
\end{equation*}
with $\om\in \S^2$, according to conservation laws of momentum and energy of two particles for an elastic collision
\begin{equation*}
u+v=u'+v',\quad 
|u|^2+|v|^2=|u'|^2+|v'|^2.
\end{equation*}
For the collision kernel $B(v-u,\om)$, it depends only on the relative velocity $|v-u|$ and $\cos\theta:=(v-u)\cdot \om/|v-u|$. Throughout the paper, we assume Grad's angular cut-off and soft potential:   
\begin{equation}\label{soft kernel}
B(v-u,\om)=|v-u|^{\gamma}q_0(\theta),\quad 
-3<\gamma<0,\quad 0\leq q_0(\t)\leq C|\cos\t |.
\end{equation}
Since we have removed angular singularity (cutoff assumption), it is convenient to use $Q_\pm(G,F)$ to denote the gain and loss terms in \eqref{def.Q}, respectively. We also consider the Boltzmann equation \eqref{def.be} with the following initial data
\begin{equation}
\label{id}
F(0,x,v)= F_{0}(x,v),\quad (x,v)\in \T^3 \times \R^3.
\end{equation}

Our aim is to construct the unique global solution $F(t,x,v)\geq 0$ of \eqref{def.be} and \eqref{id}, converging in large time to the global Maxwellian 
\Be \label{mu}
\mu(v)=\frac{1}{(2\pi)^{\frac{3}{2}}}e^{-\frac{|v|^2}{2}}
\Ee
imposing initial data $F_{0}$ which satisfies mass and energy conservation (initial compatibility condition)
\Be \label{ME_conserv}
	\iint_{\T^3\times \R^3} (F - \mu) dxdv = 0,\quad \iint_{\T^3\times \R^3} (F - \mu)|v|^2 dxdv = 0. \\
\Ee

\indent Meanwhile, one of the most important quantity in the Boltzmann theory is entropy 
\[
H(F)(t) := \iint_{\T^3\times \R^3} F\log F dxdv 
\] 
which has non-increasing in time property by celebrated H-theorem. Since we expect the global solution to converges to global Maxwellian \eqref{mu}, $\mu$ is the state with minimal entropy. To measure difference of entropy between $F$ and $\mu$, we define relative entropy 
\begin{equation}\label{Relative entropy}
\mathcal{E}(F) := \int_{\T^3} \int_{\R^3} \left(\frac{F}{\mu} \ln \frac{F}{\mu} - \frac{F}{\mu}+1 \right) \mu \, dvdx, 
\end{equation} 
which is identical to $\mathcal{E}(F)=\iint F\ln \frac{F}{\mu}\,dvdx$ under mass conservation and to $H(F)-H(\mu)$ under both mass and energy conservation \eqref{ME_conserv}. Note that the integrand in \eqref{Relative entropy} is nonnegative since it holds that $a\ln a-a+1\geq 0$ for any $a> 0$.  \\

Because of its importance, there have been lots of mathematical studies for the Boltzmann equation. In celebrated work \cite{D-Lion}, the authors proved global in time existence of renormalized solution for very general large data $F_{0} > 0$ without any smallness assumptions. However, in the aspect of statistical physics, one of the most important question for the Boltzmann equation is its asymptotic behavior, e.g., convergence to Maxwellian (equilibrium state). For this issue, \cite{DV} proved convergence to equilibrium of the Boltzmann equation using the results about Cercignani's conjecture \cite{V-Cercig} under the global in time high order regularity assumption. \\

Meanwhile, well-posedness results, including uniqueness, have been widely developed in near perturbation framework.  In \cite{Guo02, Guo03}, Guo developed high order regularity framework with small data to give well-posedness of asymptotic behavior of the equation. We also refer to \cite{RY08, RR11}. Despite its practical importance, however, there were few results about well-posedness of boundary problems of the Boltzmann equation such as specular reflection, diffuse reflection, etc. (See \cite{Mischler00}) To the best of author's knowledge, high order regularity solution of the Boltzmann equation with many physical boundary conditions are not known yet. To overcome this difficulty, Guo \cite{Guo} suggested low regularity $L^\infty$ mild solution of the boundary problem. In particular, he developed $L^{2}\mbox{-}L^{\infty}$ bootstrap argument to derive $L^{\infty}$ decay from $L^{2}$ decay. In fact, low regularity approach looks quite crucial, considering \cite{Kim11, GKTT} which imply limitedness of high order regularity solution. \\
 
The low regularity approach has been widely used after \cite{Guo}. In \cite{KL}, with his collaborator, the second author removed strong analyticity assumption of \cite{Guo} for specular reflection boundary condition case and extended the result to general $C^{3}$ convex domain. In the specular reflection boundary condition, convexity condition is fairly important to handle linear trajectory. In the following study, convexity condition was partially removed in \cite{KL-nc} and Vlasov-Poisson-Boltzmann model is also studied in \cite{CKL19}. All these works deal with hard potential cases. For soft potential case \eqref{soft kernel}, we refer \cite{SY17}, where the author treated both diffuse and specular boundary conditions. We also refer to other important works : \cite{BG} for polynomial tail (weight) and \cite{GZ} for Maxwell boundary condition.  \\

The forementioned works with low regularity approach treated small data problem, i.e., initial data $F_{0}(x,v)$ and $\mu$ are sufficiently close in $L^\infty$ sense with some velocity weight. Global well-posedness and convergence to equilibrium for large data looks extremely hard and is far beyond our knowledge. To extend small data results to broader class of initial data, \cite{DHWY17} suggested large amplitude solution of the Boltzmann equation. In their work, initial data $F_0(x,v)$ can be artificially far from $\mu$ in $L^{\infty}$ sense but still very close in $L^{1}_{x}L^{\infty}_{v}$ sense. This large amplitude result was adopted to diffuse boundary condition problem in \cite{DW} with small $L^{2}_{x,v}$ data. And recently, combining with idead of \cite{KL}, the large amplitude problem with specular reflection in $C^{3}$ convex domain with small relative entropy \eqref{Relative entropy} has been resolved in \cite{DKL20}.\\
\indent In this paper, we solve the large amplitude solution of the Boltzmann equation with soft potential in $\T^3$ domain with small relative entropy solution. We note that \cite{DHWY17} already treated $\T^3$ soft potential case, but we impose relative entropy condition instead of strong $L^{1}_{x}L^{\infty}_{v}$ condition. Moreover, the result of \cite{DHWY17} gives only decay of $f = \frac{F-\mu}{\sqrt{\mu}}$, instead of weighted $w(v)f = w(v)\frac{F-\mu}{\sqrt{\mu}}$, even though we should impose initial condition for weighted one. \\
\indent Lastly, we briefly mention about loss of velocity weight in the soft potential problem in torus $\T^{3}$. Unlike to hard potential problem, it seems natural to lose some velocity weight as we see in \cite{C08} and \cite{RY08}. For example, in \cite{RY08}, one loses whole exponential weight to obtain sub-exponential decay. Otherwise, one lose only some polynomial degree of velocity weight to get polynomial decay to Maxwellian. Using time-involved velocity weight in this paper, we lose some velocity weight, but we do not lose whole exponential weight while keeping sub-exponential decay in the case of large amplitude problem. In this sense, it is very interesting to note the result of \cite{DHWZ19} where they did not lose any weight using advantage of diffuse boundary condition! It looks that the diffuse boundary condition yields even stronger decaying property than torus. It is not clear for us, however, if the specular boundary condition also preserve whole velocity weight. Of course, we strongly expect that the time-involved weight can give large amplitude solution for the specular boundary condition, in the case of analytic boundary case, at least. \\

\subsection{Perturbation framework}
We rewrite the Boltzmann equation \eqref{def.be} in standard perturbation with Gaussian tail, $F(t,x,v)=\mu+\sqrt{\mu}f(t,x,v)$. Then, the Boltzmann equation \eqref{def.be} can be expressed as
	\begin{equation}\label{f_eqtn}
	\p_{t}f + v\cdot \nabla_{x} f+Lf  
	= \Gamma(f,f) , 
	\end{equation} 
with the usual notations on the linearized operator 
 \begin{equation} \label{L_operator}
 \begin{split}
 	Lf &:= - \frac{1}{\sqrt{\mu}}\{ Q(\mu, \sqrt{\mu}f) + Q(\sqrt{\mu}f,\mu)\} = \nu(v) f -Kf,  \\
 	Kf &= - \frac{1}{\sqrt{\mu}}\{ Q_{+}(\mu, \sqrt{\mu}f) + Q_{+}(\sqrt{\mu}f,\mu) + Q_{-}(\sqrt{\mu}f,\mu)\} := \int_{\R_{u}^{3}} k(v,u)f(u) du,
 \end{split}
 \end{equation}
 (we abbreviated time and space variable of $f$ for simplicity) and the quadratic nonlinear term
 \begin{equation*}
     \Gamma(g,f) := \frac{1}{\sqrt{\mu}}Q(\sqrt{\mu} g, \sqrt{\mu}f).
 \end{equation*}
Here, we have used $\nu(v)$ to denote the collision frequency 
\begin{equation}
\label{def.nu}
\nu (v) :=\int_{\R^3}\int_{\S^2} \vert v-u \vert ^\gamma q_0(\theta) \mu(u) \, d\omega du,\quad -3< \gamma < 0.
\end{equation}
It is well-known that there exist generic positive constants $C_{1,\gamma},C_{2,\gamma}>0$ such that 
 \begin{equation} \label{nu_est}
 	C_{1,\gamma} (1+\vert v \vert ^2)^{\gamma/2} \leq \nu(v) \leq C_{2,\gamma} (1+\vert v\vert^2)^{\gamma/2},\quad -3< \gamma < 0. 
 \end{equation}
 The integral operator $Kf$ in \eqref{L_operator} can be further splitted into $Kf=K_2f-K_1f$ with 
\begin{align} \label{k_1 and k_2}
	\begin{split}
		(K_1f)(v) &= \int_{\R^3 \times \S^2} \vert v-u \vert^{\gamma}q_0(\theta) \sqrt{\mu(u)}\sqrt{\mu(v)}f(u) \, d\omega du := \int_{\R^3} k_1(v,u)f(u) \,d u, \\ 
		(K_2f)(v)&= \int_{\R^3 \times \S^2} \vert v-u \vert^{\gamma}q_0(\theta) \sqrt{\mu(u)}\left[ \sqrt{\mu(v')}f(u') + \sqrt{\mu(u')}f(v') \right]\,d\omega du := \int_{\R^3} k_2(v,u)f(u) \,du,
	\end{split}
\end{align}
where $k(v,u) := k_2(v,u)-k_1(v,u)$ by \eqref{L_operator}. \\ 
We also note that \eqref{def.be} satisfy mass and energy conservations \eqref{ME_conserv}, and in our perturbation framework, those are rewritten by
\begin{equation} \label{conservation}
    \int_{\T^3} \int_{\R^3} \sqrt{\mu}f(t,x,v)  \, dvdx \, =\, 0, \quad 
    \int_{\T^3} \int_{\R^3} \vert v\vert^2 \sqrt{\mu}f(t,x,v) \, dvdx \, =\, 0,\quad t\geq 0.  \\
\end{equation}

\subsection{Main result and scheme of the proof}
We solve \eqref{f_eqtn} with \eqref{soft kernel} in torus domain $\T^3$ for given initial data $f_{0}$. The following theorem is the main result of this paper.
\begin{theorem} \label{mainthm}
	Let $(t,x,v) \in \R_+ \times \T_{x}^3 \times \R_{v}^3$ and  assume initial data $f_0$ satisfies \eqref{conservation}. Define a weight function  
	\begin{equation} \label{t-weight}
	w_{q,\vartheta,\beta}(v,t) = (1+\vert v \vert^2)^{\beta}\exp\left\{\frac{q}{8}\left(1+\frac{1}{(1+t)^{\vartheta}}\right)|v|^2\right\},
	\end{equation}
	where 
		$\beta\geq7/2$,  $ 0\leq \vartheta < -\frac{2}{\gamma}$ and $q$ satisfies \eqref{q}.
	Then for any $M_0>0$, there exists $\epsilon_0 >0$ such that if $f_0$ satisfies 
	\begin{equation}
	\|w_{q,\vartheta,\beta}f_0\|_{L^{\infty}}\le M_0, \quad \mathcal{E}(F_0)\le \epsilon_0,
	\end{equation}
	there exists a unique global solution $F(t,x,v)=\mu(v) + \sqrt{\mu(v)}f(t,x,v)\ge 0$ to Boltzmann equation \eqref{def.be} with soft potential \eqref{soft kernel} on $\T^3_x \times \R^3_v$. Moreover, $f$ satiesfies
	\begin{align*}
	\|w_{q.\vartheta,\beta}f(t)\|_{L^{\infty}} \le \left(4CM_0\exp\left \{\frac{4CM_0}{\lambda_0\rho}\right \} + 4M_0\right)e^{-\lambda_0 t^{\rho}},
	\end{align*}
	for some $\lambda_0>0$, where $\rho-1= \frac{(1+\vartheta)\gamma}{2-\gamma} > 0$.  \\
\end{theorem}

One of the main difference to hard potential is behavior of collision frequency. Let us recall uniform estimate of collision frequency
\Be \label{freq est}
	\nu(v) := \iint_{\T^3\times \R^3} B(v-u, \o) \mu(u) d\o du \simeq \langle v \rangle^{\g} := (1 + |v|^2)^{\frac{\g}{2}},\quad v\in \R^3.
\Ee
If we read \eqref{f_eqtn} with \eqref{L_operator} along linear trajectory 
\[
	X(s;t,x,v) = x-v(t-s),\quad V(s;t,x,v)=v,
\]
(Here, $X(s;t,x,v)$ and $V(s;t,x,v)$ mean position and velocity of a particle at time $s$ which was at $(t,x,v)$),
we have linearized exponential decay $e^{-\nu(v)t}$ along the trajectory. In the case of hard potential $0\leq \g \leq 1$, \eqref{freq est} has a uniformly positive lower bound and hence we get uniform exponential decay factor. However, with soft potential $-3 < \g < 0$,  \eqref{freq est} is not uniformly positive anymore since $|v| \rightarrow \infty$. To resolve this problem without abandoning weight, we adopt time-involved velocity weight 
\begin{equation} \label{weight}
w_{q,\vartheta,\b}(v)= (1+\abs{v}^2)^{\beta}\exp\left\{\frac{q}{8}(1+(1+t)^{-\vartheta}) \abs{v}^2 \right \}, \quad \beta\geq7/2,\quad  0<q<1, \quad 0\leq \vartheta < -\frac{2}{\gamma},
\end{equation}
which was introduced in \cite{RTH12,RSTH13,RTH13,R14}. Note that exponent factor $\frac{q}{8}(1+(1+t)^{-\vartheta}) $ is uniform for all $0\leq t < \infty$. Putting time dependence we get extra coercive factor in the equation of $h := w_{q,\vartheta,\b}f$ :
\[
	\p_{t}h + v\cdot\nabla h + \Big( \nu(v) +\frac{\vartheta q \vert v \vert^2}{8(1+t)^{\vartheta+1}} \Big)h = w_{q,\vartheta,\b} Kf + w_{q,\vartheta,\b}\Gamma(f,f).
\]
Applying Young's inequality, we are able to deduce a lower bound depending on time for a new collision frequency $\tilde{\nu}(t,v)$ 
\begin{equation} \label{nutilde} 
\tilde{\nu}(t,v) := \nu(v) +\frac{\vartheta q \vert v \vert^2}{8(1+t)^{\vartheta+1}} \geq C_{q,\vartheta} (1+t)^\frac{(1+\vartheta)\gamma}{2-\gamma},
\end{equation}
by \eqref{nutilde estimate}
which implies the subexponential time decay property. Specifically, from $-3<\gamma<0$ and $0\leq \vartheta <-\frac{2}{\gamma}$, 
\begin{equation*}
0> \frac{(1+\vartheta) \gamma}{2-\gamma} > \frac{\gamma-2}{2-\gamma} = -1.  \\
\end{equation*}

In large amplitude problem, we cannot bound nonlinear quadratic $\Gamma(f,f)$ as $\|wf\|_{L^\infty}^{2}$ since $L^\infty$ could be artificially large in general. Instead, we perform $L^\infty\times L^{p}$ type estimate 
\Be \label{infty_p}
	|w\Gamma_{\pm}(f,f)| \lesssim \langle v \rangle^{\g} \|wf\|_{\infty} \Big( \int \mathfrak{k}(u) |wf(u)|^{p} du\Big)^{\frac{1}{p}},
\Ee
for nonlinear $\Gamma(f,f)$ where $\mathfrak{k}(u)$ has proper decay. (Lemma \ref{gamma estimate}.) Unlike to hard potential case (\cite{DW}, \cite{DKL20}), however, $\Gamma_{-}$ can be treated in another way.  Nonlinear $\Gamma_{-}$ contains local term $f(v)$, so we see factor $\langle v \rangle^{\g}$. In the case of hard potential this gives growth in $|v|$ so $\Gamma_{-}$ was combined with collision frequency and treated separately. In soft potential case, $\langle v \rangle^{\g}$ gives polynomial decay in $|v|$ and there is no big problem in dealing with $\Gamma_{-}$.  \\

Once we have $L^\infty\times L^{p}$ type estimate \eqref{infty_p}, we can treat $\Gamma_{\pm}$ terms similar as linear $Kf$ except  for extra $\|wf\|_{\infty}$. Motivated by Duhamel iteration and $L^{2}\mbox{-}L^{\infty}$ scheme, we will derive $\|f\|_{L^{2}_{x,v}}$ in double Duhamel iteration. Since relative entropy control $L^{2}$ and $L^{1}$ of $f$ partially, (see Lemma \ref{entropy_est}) we can close proper apriori estimate and bootstrap argument allows us to extend the time interval until $\|wf\|_{\infty}$ becomes sufficiently small so that small data theory becomes valid.  \\

Lastly, we briefly comment about some delicate issues of treating kernel $k(v,u)$ of linear operator $K$ in \eqref{L_operator}. Because of singular behavior of kernel near $v\simeq u$ ($|v-u|^{\g}$ is not square integrable locally), we adopt cutoff kernel \eqref{K-chi} \cite{RY08} to treat singular region separately,
\[
	k(v,u) = k^{\chi}(v,u) + k^{1-\chi}(v,u)
\]
where $k^{\chi}(v,u)$ has support away from singularity and $k^{1-\chi}(v,u)$ corresponds to very small region near singularity. When we perform Duhamel iteration twice, the most delicate issue comes when we have interaction between $k^{\chi}(v,u)$ in the first iteration and $k^{1-\chi}(u,u^{\prime})$ in the second iteration. (We have similar issue when we play with $\Gamma$, but by structure \eqref{infty_p}, it will be resolved by similar manner.) In fact we have to control
\Be \label{issue}
	\int_0^t e^{-\int_s^t \tilde{\nu}(v,\tau)\,d\tau}\int_u k_w^{\chi}(v,u) \int_0^s e^{-\int_{s'}^s \tilde{\nu}(u,\tau)\,d\tau} \int_{u'} k_w^{1-\chi}(u,u')h(u')
\Ee
and, unlike hard potential case, time integration does not help us, because of polynomial growth ($-3 < \g < 0$)
\begin{align*}
\int_0^t e^{-\int_s^t \tilde{\nu}(v,\tau)\,d\tau} \sim \nu^{-1}(v) \quad \textrm{and} \quad \int_0^s e^{-\int_{s'}^s \tilde{\nu}(u,\tau)\,d\tau} \sim \nu^{-1}(u).
\end{align*}
Fortunately, from \eqref{k_esti.2}, we obtain Gaussian decaying factor and we can obtain small control 
\begin{align*}
\eqref{issue} &\sim \nu^{-1}(v) \int_u k_w^{\chi}(v,u) \nu^{-1}(u) \int_{u'}k_w^{1-\chi}(u,u') h(u') \\
&\stackrel{\text{u' integration}}{\sim} \epsilon^{\gamma+3}\Vert h \Vert_{L^\infty} \nu^{-1}(v) \int_u k(v,u) \frac{(1+ \vert v \vert^2)^{2\beta}}{(1+\vert u \vert ^2)^{2\beta}} \frac{\textcolor{black}{(1+\vert u \vert^2)^\beta}}{(1+\vert v \vert^2)^{\beta}} \textcolor{black}{\nu^{-1}(u)\mu(u)}  \\
&\stackrel{\text{u integration}}{\sim} \epsilon^{\gamma+3} \Vert h \Vert_{L^\infty} \textcolor{black} { \frac{\nu^{-1}(v)}{(1+\vert v \vert ^2)^{\beta}}}\\
& \sim \textcolor{black}{\epsilon^{\gamma+3}} \Vert h \Vert_{L^\infty}.  \\
\end{align*}

Additionally, another issue occurs from combination between $\Gamma$ in the first iteration and $K$ in the second iteration. Although we adopt the cutoff kernel $k_w^\chi$, we can only guarantee that the cutoff kernel is square integrable in Lemma \ref{Lemma_k.0}. Unfortunately, $p$ is greater than 5 in $L^\infty\times L^p$ type estimate \eqref{infty_p}, which makes the part $\Gamma\mbox{-}K$ difficult to handle. To deal with this difficulty, we used the fact that the cutoff kernel is square integrable and H\"older's inequality. Hence, we obtain the following bound:
\begin{align*}
	&\Vert h\Vert_{L^\infty} \left (\int_{u}\mathfrak{k}(u) \left(\int_{u'}k_w^{\chi}(u,u') \vert h(u')\vert \, du'\right)^p \,du \right)^{1/p} \\
	&\quad \sim \Vert h \Vert_{L^\infty}\left(\int_u \mathfrak{k}(u) \left(\left(\int_{u'} \vert k_w^{\chi}(u,u') \vert^2 \, du'\right)^{1/2} \left(\int_{u'} \vert h(u')\vert ^2 \,du' \right)^{1/2} \,du\right)^{p}\right)^{1/p}\\
	&\quad \sim \Vert h \Vert_{L^\infty} \left( \int_u \int_{u'} \mathfrak{k}(u) \vert h(u') \vert^p \,du'du\right)^{1/p}\\
	&\quad \sim \Vert h \Vert_{L^\infty}^{2-\frac{2}{p}} \mathcal{E}(F_0).
\end{align*}
Throughout this paper, we shall use the notation $\Vert \cdot \Vert_{L^\infty_{x,v}}:= \Vert \cdot \Vert_{L^\infty}$. 
\section {preliminaries}
\subsection{{Operator $K$}}
Recall the definition of operator $K:=K_2-K_1$
\begin{align} \label{Kf}
	(Kf)(v)&= \int_{\R^3} k(v,u) f(u)\, du =\int_{\R^3} k_2(v,u) f (u) \, du - \int_{\R^3} k_1 (v,u)f(u) \, du,
\end{align} 
where $k_i(v,u)$ is a symmetric integral kernel of $K_i$ for $i=1,2$. 
\begin{lemma} \label{Lemma_k.2} \cite{DHWY17} The integral kernel $k_1$ and $k_2$ of \eqref{Kf} satisfy 
\begin{equation}
	0\leq k_1(v,u) \leq C \vert v-u \vert ^\gamma e^{-\frac{\vert v \vert^2}{4}} e^{-\frac{\vert u \vert ^2}{4}}, 
\end{equation}
and 
\begin{equation}
	0\leq k_2(v,u) \leq \frac{C_\gamma}{\vert v -u \vert^{\frac{3-\gamma}{2}}} e^{-\frac{\vert v-u\vert^2}{8}} e^{-\frac{\vert \vert v \vert^2 - \vert u \vert^2 \vert^2}{8\vert v-u\vert^2}},
\end{equation}
where $C>0$ is a generic constant and $C_{\gamma}$ is a constant depending on $\gamma$. 
\end{lemma}

To treating the singularity of $K$, we introduce modified kernel with smooth cutoff: 
\begin{equation} \label{cutoff}
	\chi(\vert v-u \vert) =
	\begin{cases}
	1, \quad \textrm{if} \quad\vert v-u \vert \geq 2\epsilon, \\ 
	0, \quad \textrm{if} \quad \vert v-u \vert \leq \epsilon. 
	\end{cases}
\end{equation}
We split the operator $K$ using cutoff function $\chi$ :
\begin{equation*}
	Kf=K^{\chi}f + K^{1-\chi}f, \quad K_2f=K_2^{\chi}f+K_2^{1-\chi}f, \quad \textrm{and} \quad K_1f=K_1^{\chi} f + K_1^{1-\chi}f.
\end{equation*}
Specifically,  for $i=1,2$, we define 
\begin{align} \label{K-chi}
	\begin{split}
		K^{\chi} f &= \int_{\R^3} \chi(\vert v-u \vert) k(v,u) f(u) \, du = \int_{\R^3} k^{\chi}(v,u)f(u) \, du,\\ 
		K^{1-\chi} f &= \int_{\R^3} (1-\chi(\vert v-u \vert) k(v,u) f(u) \, du = \int_{\R^3} k^{1-\chi} (v,u) f(u) du, \\ 
		K_i^{\chi} f &= \int_{\R^3} \chi(\vert v-u \vert) k_i(v,u) f(u) = \int_{\R^3} k_i^{\chi} (v,u) f(u) \, du, \\
		K_i^{1-\chi}f &= \int_{\R^3} (1-\chi(\vert v-u \vert) k_i(v,u) f(u) \, du =\int_{\R^3} k_i^{1-\chi}(v,u) f(u) \, du.
	\end{split}
\end{align}
\begin{lemma} \label{Lemma_k.0}
\cite{RY08} There are constants $C>0$ and $C_{\epsilon}>0$ depending on $\epsilon$ such that 
\begin{equation} \label{k2_esti.0}
	\vert k_2^{\chi}(v,u)\vert \leq C \epsilon^{\gamma-1} \frac{\exp\left(-\frac{1}{8} \vert u - v\vert^2 -\frac{1}{8} \frac{(\vert v\vert^2-\vert u \vert^2)^2}{\vert v-u \vert^2}\right)}{\vert v -u \vert}, 
\end{equation}	
or 
\begin{equation} \label{k2_esti.1}
	\vert k_2^{\chi}(v,u)\vert \leq C_{\epsilon} \frac{\exp \left(-\frac{s_2}{8}\vert v- u\vert^2 -\frac{s_1}{8} \frac{(\vert v \vert^2 -\vert u \vert^2)^2}{\vert v -u \vert^2} \right) } {\vert v-u \vert (1+\vert v \vert + \vert u \vert)^{1-\gamma}},
\end{equation}
for any $0<s_1<s_2<1$. 
\end{lemma}

With a weight introduced in \eqref{t-weight}, we have the following linearized opeator estimates.
\begin{lemma} \label{Lemma_k.1}
\cite{SY17}
There exists a constant $C>0$ such that
\begin{equation} \label{k_esti.2}
	 w_{q,\vartheta,\beta} (v) K^{1-\chi}f \leq C \mu(v)^{\frac{1-q}{8}} \epsilon^{\gamma+3} \Vert w_{q,\vartheta,\beta} f \Vert_{L^\infty}, 
\end{equation} 
and
\begin{equation} \label{k_esti.0}
	w_{q,\vartheta,\beta}(v) \int_{\R^3} k^{\chi} (v,u) e^{\epsilon \vert v-u \vert^2} \vert f(u) \vert \, du \leq C_{q,\epsilon} \langle v \rangle^{\gamma-2} \Vert w_{q,\vartheta,\beta} f \Vert_{L^\infty}. 
\end{equation}
 
\end{lemma}
\begin{proof}
Note that
\begin{equation} \label{K_split.1}
	w_{q,\vartheta,\beta}K^{1-\chi}f = w_{q,\vartheta,\beta}K^{1-\chi}_1f+w_{q,\vartheta,\beta} K^{1-\chi}_2f.
\end{equation}  
We firstly consider $K^{1-\chi}_1$ in \eqref{K_split.1}.
\begin{align} \label{K_1 esti}
	w_{q,\vartheta,\beta}(v)K_1^{1-\chi} f &\leq  C w_{q,\vartheta,\beta}(v)\sqrt{\mu(v)} \int_{\vert v-u \vert \leq 2\epsilon} 
	\vert v-u \vert^{\gamma} \sqrt{\mu(u)} \frac{\vert w_{q,\vartheta,\beta}f(u)\vert }{w_{q,\vartheta,\beta}(u)} \, du\nonumber \\
	&\leq C(1+\abs{v}^2)^{\beta}\exp\left(\frac{q}{8}(1+(1+t)^{-\vartheta}) \vert v \vert^{2}\right)\sqrt{\mu(v)} \epsilon^{\gamma+3}\Vert w_{q,\vartheta,\beta}f(t) \Vert_{L^\infty}\nonumber \\
	&\leq C \mu(v)^{\frac{1-q}{8}}\epsilon^{\gamma+3} \Vert w_{q,\vartheta,\beta} f \Vert_{L^\infty}.
\end{align}
Next, we will deal with the remaining part $K_2^{1-\chi}$ in \eqref{K_split.1}.
\begin{equation}
	w_{q,\vartheta,\beta}(v)K_2^{1-\chi} f \leq Cw_{q,\vartheta,\beta}(v) \int_{\vert v-u \vert \leq 2\epsilon} \vert v-u \vert^{\gamma} \sqrt{\mu(u)}\left[ \sqrt{\mu(u')}\frac{\vert w_{q,\vartheta,\beta}f(v')\vert}{w_{q,\vartheta,\beta}(v')} +\sqrt{\mu(v')} \frac{\vert w_{q,\vartheta,\beta}f(u')\vert}{w_{q,\vartheta,\beta}(u')} \right] \, du.
\end{equation} 
Under $\vert v-u \vert \leq 2\epsilon$, 
\begin{align*} \label{w.2}
	\vert v' \vert &= \vert v + [(u-v) \cdot \omega] \omega \vert \geq \vert v \vert - \vert v-u \vert \geq \vert v \vert - 2\epsilon, \\ 
	\vert u' \vert &= \vert v+u-v-[(u-v) \cdot \omega]\omega \vert \geq \vert v \vert -2 \vert v-u \vert \geq \vert v \vert - 4 \epsilon, \\ 
	\vert u \vert &= \vert v+u-v \vert \geq \vert v \vert - \vert v-u \vert \geq \vert v \vert -2\epsilon,
\end{align*}
which implies 
\begin{align}
	\begin{cases}
	\sqrt{\mu(u)}\sqrt{\mu(u')} &\leq e^{-\frac{ (\vert v \vert -2\epsilon)^2}{4}} e^{- \frac{(\vert v \vert -4\epsilon)^2}{4}} \leq e^{-\frac{\vert v \vert^2}{2}} e^{3\epsilon \vert v \vert}  \leq C \sqrt{\mu(v)}, \\ 
	\sqrt{\mu(u)}\sqrt{\mu(v')} &\leq e^{-\frac{ (\vert v \vert -2\epsilon)^2}{4}} e^{- \frac{(\vert v \vert -2\epsilon)^2}{4}}\leq e^{-\frac{\vert v \vert^2}{2}} e^{2\epsilon \vert v \vert}  \leq C \sqrt{\mu(v)},
	\end{cases}
\end{align}
where we used $\epsilon<1$ and the following fact 
\begin{equation*}
		C\epsilon \vert v \vert =\left(\frac{\vert v \vert}{2} \right)\left(2C\epsilon\right) \leq \frac{\vert v \vert^2}{4} + C_{\epsilon}.
\end{equation*} 
From \eqref{w.2}, it holds that 
\begin{equation} \label{K_2 esti}
	w_{q,\vartheta,\beta}(v) K_2^{1-\chi} \left( \frac{\vert w_{q,\vartheta,\beta}f\vert }{w_{q,\vartheta,\beta}}\right) (t,x,v) \leq C_{q} \mu(v)^{\frac{ 1-q }{8}}\epsilon^{\gamma+3} \Vert w_{q,\vartheta,\beta}f(t) \Vert_{L^\infty}.
\end{equation}
Summing \eqref{K_split.1},\eqref{K_1 esti}, and \eqref{K_2 esti} yields 
\begin{equation}
	w_{q,\vartheta,\beta}(v) K^{1-\chi}f \leq  C_{q} \mu(v)^{\frac{1-q}{8}} \epsilon^{\gamma+3} \Vert w_{q,\vartheta,\beta}f(t) \Vert_{L^\infty}.
\end{equation}
For \eqref{k_esti.0}, we firstly consider the part $K_1^{\chi}$: 
\begin{align*}
	&w_{q,\vartheta,\beta}(v) \int_{\R^3} \textbf{k}_1^{\chi}(v,u) e^{\epsilon \vert v-u\vert^2} \vert f(u) \vert  \,du \\ 
	&\quad \leq C w_{q,\vartheta,\beta}(v) \sqrt{\mu(v)} \int_{\R^3} \chi(\vert v-u\vert) \vert v-u \vert^{\gamma} \sqrt{\mu(u)} \frac{e^{\epsilon \vert v-u\vert^2} \vert w_{q,\vartheta,\beta}f(u) \vert}{w_{q,\vartheta,\beta}(u)}\, du\\
	&\quad \leq C_{q} \int_{\R^3} \vert v-u\vert^\gamma \mu(v) ^{\frac{ 1-q}{8}} \sqrt{\mu(u)} e^{\epsilon \vert v-u\vert^2} \, du \Vert w_{q,\vartheta,\beta}f(t) \Vert_{L^\infty}\\
	&\quad \leq C_{q} \langle v \rangle^{\gamma} \mu(v)^{\frac{ 1-q }{16}}\Vert w_{q,\vartheta,\beta}f(t)\Vert_{L^\infty},
\end{align*}	
where the last inequality comes from $\displaystyle \int_{\R^3} \vert v-u \vert^{\gamma} \mu(u)^{1/4} \,du \leq C \langle v \rangle^\gamma$. \newline
It remains to check the part $K_2^\chi$ for \eqref{k_esti.0}. Recall \eqref{k2_esti.1} and take $s_0=\min\{s_1,s_2\}$. Then,
\begin{align} \label{K_2 esti.1}
	&w_{q,\vartheta,\beta}(v) \int_{\R^3} \textbf{k}_2^\chi (v,u) \left(\frac{e^{\epsilon \vert v-u \vert ^2 }\vert w_{q,\vartheta,\beta}f(u)\vert}{w_{q,\vartheta,\beta}(u)}\right) \,du \nonumber \\ 
	&\quad \leq C_{\epsilon} \Vert w_{q,\vartheta,\beta}f(t) \Vert_{L^\infty} \langle v \rangle^{\gamma-1} w_{q,\vartheta,\beta}(v) \int_{\R^3} \frac{\exp\left(-\frac{s_0}{8} \vert v-u \vert^2 -\frac{s_0}{8} \frac{(\vert v \vert^2 - \vert u \vert^2)^2}{\vert v-u \vert^2} \right)}{\vert v-u \vert} \frac{e^{\epsilon\vert v-u\vert ^2}}{w_{q,\vartheta,\beta}(u)}\,du.
\end{align}
From definition of the weight function \eqref{weight}, we can deduce that 
\begin{equation*}
	\frac{w_{q,\vartheta,\beta}(v)}{w_{q,\vartheta,\beta}(u)} \leq C_{\beta}(1+\abs{v-u}^2)^{\beta}e^{-\frac{\tilde{q}}{4}(\vert u \vert^2 - \vert v \vert ^2)},
\end{equation*}
where $\tilde{q}=\frac{q}{2}\left(1+(1+t)^{-\vartheta}\right)$. Notice that $\frac{q}{2} < \tilde{q} \leq q$. Let $v-u=\eta$ and $u=v-\eta$ in the integral of \eqref{K_2 esti.1}. We now compute the total exponent in $\textbf{k}_2^\chi(v,u) \frac{w_{q,\vartheta,\beta}(v)}{w_{q,\vartheta,\beta}(u)}$: 
\begin{align*}
	&-\frac{s_0}{8} \vert \eta \vert^2 -\frac{s_0}{8} \frac{\vert \eta \vert^2 - 2v \cdot \eta \vert^2}{\vert \eta \vert ^2} - \frac{\tilde{q}}{4} ( \vert v -\eta \vert^2 -\vert v \vert^2)\\
	&\quad =-\frac{s_0}{4} \vert \eta \vert^2 +\frac{s_0}{2} v\cdot \eta - \frac{s_0}{2} \frac{ \vert v \cdot \eta \vert^2}{\vert \eta \vert^2} -\frac{ \tilde{q}}{4}(\vert \eta \vert ^2 -2v\cdot \eta)\\
	&\quad = -\frac{1}{4}(\tilde{q}+s_0) \vert \eta \vert ^2 +\frac{1}{2} (s_0+\tilde{q}) v\cdot \eta - \frac{s_0}{2} \frac{(v\cdot \eta)^2}{ \vert \eta \vert ^2}. 
\end{align*}
Setting
\begin{equation} \label{q} 
0<\tilde{q} \leq q < s_0<1,
\end{equation}
where $\tilde{q}=\frac{q}{2}\left(1+(1+t)^{-\vartheta}\right)$ and $s_0=\min\{s_1,s_2\}$, the discriminant of the above quadratic form of $\vert \eta \vert $ and $\frac{v\cdot \eta}{\vert \eta \vert}$ is 
\begin{equation*}
	\Delta = \frac{1}{4}(s_0+\tilde{q})^2 - (\tilde{q}+ s_0)\frac{s_0}{2} = \frac{1}{4}(\tilde{q}^2 -s_0^2) <0. 
\end{equation*}
Thus, we have, for $\epsilon>0$ sufficiently small and $q<s_0$, that there is $C_q>0$ such that 
\begin{align} \label{e.1}
	\begin{split}
	&-\frac{s_0-8\epsilon}{8} \vert \eta \vert^2 -\frac{s_0}{8} \frac{\vert \vert \eta \vert ^2 -2v\cdot \eta \vert^2}{\vert \eta \vert^2} - \frac{\tilde{q}}{4} (\vert \eta \vert^2 -2v\cdot \eta)  \\ 
	&\quad \leq -C_q \left \{ \vert \eta \vert^2 +\frac{\vert v \cdot \eta \vert^2}{\vert \eta \vert^2} \right \} \\ 
	&\quad = -C_q\left \{ \frac{\vert \eta \vert^2}{2} +\left(\frac{\vert \eta \vert^2}{2}+\frac{\vert v \cdot \eta \vert^2}{\vert \eta \vert^2}\right) \right \} \\
	&\quad \leq -C_q \left \{ \frac{\vert \eta \vert^2}{2} + \vert v \cdot \eta \vert \right \}.
	\end{split}
\end{align}

Substituting \eqref{e.1}  into \eqref{K_2 esti.1}, one then has 
\begin{align}\label{e.3}
	\begin{split}
	&\Vert w_{q,\vartheta,\beta}f(t) \Vert_{L^\infty} \langle v \rangle^{\gamma-1} w_{q,\vartheta,\beta}(v) \int_{\R^3} \frac{\exp\left(-\frac{s_0}{8} \vert v-u \vert^2 -\frac{s_0}{8} \frac{(\vert v \vert^2 - \vert u \vert^2)^2}{\vert v-u \vert^2} \right)}{\vert v-u \vert} \frac{e^{\epsilon\vert v-u\vert ^2}}{w_{q,\vartheta,\beta}(u)}\,du \\
	&\quad \leq C_{q,\epsilon} \langle v \rangle^{\gamma-1} \Vert w_{q,\vartheta,\beta}f(t) \Vert_{L^\infty} \int_{\R^3} \frac{(1+\vert \eta \vert^2)^\beta}{\vert \eta \vert} \exp\left \{ -C_q \left \{ \frac{\vert \eta \vert^2}{2}+ \vert v \cdot \eta \vert \right \} \right\} \,d\eta\\
	&\quad \leq C_{q,\epsilon} \langle v \rangle^{\gamma-1} \Vert w_{q,\vartheta,\beta}f(t) \Vert_{L^\infty} \int_{\R^3}\frac{1}{\vert \eta \vert} \exp\left \{ -C_q \left \{ \frac{\vert \eta \vert^2}{4}+ \vert v \cdot \eta \vert \right \} \right\} \,d\eta.
	\end{split}
\end{align} 
For $\vert v \vert \geq 1 $, we change the variables $\eta_{\parallel}= \left\{ \eta \cdot \frac{v}{\vert v \vert} \right \} \frac{v}{\vert v \vert}$, and $\eta_\perp = \eta -\eta_{\parallel}$ so that $ \vert v \cdot \eta \vert = \vert v \vert \times \vert \eta_{\parallel}\vert$, which implies that 
\begin{align}\label{e.4}
	\begin{split}
 &\int_{\R^3} \frac{1}{\vert \eta \vert} \exp\left \{ -C_q \left \{ \frac{\vert \eta \vert^2}{4}+ \vert v \cdot \eta \vert \right \} \right\} \,d\eta \\ 
 &\quad \leq C \int_{\R^2} \frac{1}{\vert \eta_{\perp} \vert} e^{-C_q \frac{\vert \eta \vert^2}{4}} \left \{\int_{-\infty}^\infty e^{-C_q \vert v \vert \times \vert \eta_{\parallel} \vert} \, d \vert \eta_{\parallel} \vert \right\}\, d\eta_{\perp}\\
 &\quad \leq \frac{C}{\vert v \vert} \int_{\R^2} \frac{1}{\vert \eta_\perp\vert} e^{-\frac{C_q}{4} \vert \eta_\perp \vert^2} \left \{ \int_{-\infty}^\infty e^{-C_q \vert y \vert} dy\right \} \, d \eta_{\perp} \quad (y= \vert v \vert \times \vert \eta_{\parallel}\vert)\\
 &\quad \leq \frac{C}{1+\vert v \vert}. 
 	\end{split}
\end{align}
On the other hand, for $\vert v \vert \leq 1$, 
\begin{align*}
&\int_{\R^3} \frac{1}{\vert \eta \vert} \exp\left \{ -C_q \left \{ \frac{\vert \eta \vert^2}{4}+ \vert v \cdot \eta \vert \right \} \right\} \,d\eta \\ 
 &\quad \leq C \int_{\R^2} \frac{1}{\vert \eta_{\perp} \vert} e^{-C_q \frac{\vert \eta_\perp \vert^2}{4}} \left \{\int_{-\infty}^\infty e^{-\frac{C_q}{4}\vert \eta_{\parallel}\vert^2}e^{-C_q \vert v \vert \times \vert \eta_{\parallel} \vert} \, d \vert \eta_{\parallel} \vert \right\}\, d\eta_{\perp}\\
 &\quad \leq C \int_{\R^2} \frac{1}{\vert \eta_\perp\vert} e^{-\frac{C_q}{4} \vert \eta_\perp \vert^2} \left \{ \int_{-\infty}^\infty e^{-\frac{C_q}{4}\vert \eta_{\parallel}\vert^2} d\vert\eta_{\parallel}\vert\right \} \, d \eta_{\perp} \quad \\
 &\quad \leq \frac{C}{1+\vert v \vert}. 
\end{align*}
Summing \eqref{K_2 esti.1}, \eqref{e.3} and \eqref{e.4}, we obtain 
\begin{equation}
	w_{q,\vartheta,\beta}(v) \int_{\R^3} \textbf{k}_2^\chi (v,u) \left( \frac{e^{\epsilon \vert v-u \vert^2}\vert w_{q,\vartheta,\beta}f(u) \vert}{w_{q,\vartheta,\beta}(u)}\right) \,du \leq C_{q,\epsilon} \langle v \rangle^{\gamma-2} \Vert w_{q,\vartheta,\beta}f(t) \Vert_{L^\infty}.
\end{equation}
Thus, we complete the proof of Lemma \ref{Lemma_k.1}.
\end{proof}

\subsection{Nonlinear term $\Gamma(f,f)$}
We control nonlinear $\Gamma_{\pm}$ by a product form of $L^{\infty}$ and $L^{p}$.
\begin{lemma} \label{gamma estimate} \cite{DW} Let $p$ be a positive number satisfying 
\begin{equation} \label{cond.p}
	p>1,\, p\gamma >-3. 
\end{equation}
There are constants $C_\gamma$ depending only on $\gamma$ such that 
\begin{equation} \label{gamma_loss}
	\left \vert w_{q,\vartheta,\beta}(v) \Gamma_-(f,f)(v) \right \vert \leq C_\gamma \nu(v) \Vert w_{q,\vartheta,\beta} f(t) \Vert_{L^\infty} \left ( \int_{\R^3} \vert f(u)\vert ^{p'} \,du \right)^{1/p'},
\end{equation}
and
\begin{equation} \label{gamma_gain}
	\left \vert w_{q,\vartheta,\beta}(v) \Gamma_+(f,f)(v) \right \vert \leq C_\gamma \nu(v) \Vert w_{q,\vartheta,\beta} f(t) \Vert_{L^\infty} \left ( \int_{\R^3} (1+\vert u \vert)^{-2\beta p'+16} \vert w_{q,\vartheta,\beta}f(u)\vert ^{p'} \,du \right)^{1/p'},
\end{equation}
where $p'=\frac{5p}{p-1}$.  
\end{lemma}

\begin{proof}
By the definition of $\Gamma_-(f,f)$, it holds that 
\begin{equation} \label{p.1}
	\left \vert w_{q,\vartheta,\beta}(v) \Gamma_-(f,f)(v)  \right \vert \leq C \Vert w_{q,\vartheta,\beta}f\Vert_{L^\infty}\int_{\R^3} \vert v-u \vert^{\gamma} \sqrt{\mu(u)} \vert f(u) \vert \, du. 
\end{equation} 
From \eqref{cond.p} and \eqref{p.1} and H\"{o}lder's inequality, it holds that 
\begin{align}
	&\left \vert w_{q,\vartheta,\beta}(v) \Gamma_-(f,f)(v) \right \vert \nonumber \\ 
	& \quad \leq C \Vert w_{q,\vartheta,\beta} f \Vert_{L^\infty} \left( \int_{\R^3} \vert v-u \vert^{p\gamma} \sqrt{\mu(u)}\,du\right)^{1/p} \left ( \int_{\R^3} \sqrt{\mu(u)} \vert f(u) \vert^{\frac{p}{p-1}} \, du \right)^{1-\frac{1}{p}} \nonumber \\
	& \quad \leq C \nu(v)\Vert w_{q,\vartheta,\beta} f \Vert_{L^\infty} \left ( \int_{\R^3} (\mu(u))^{5/8} \, du \right )^{\frac{4}{5}\left(1-\frac{1}{p}\right)} \left ( \int_{\R^3} \vert f(u) \vert ^{\frac{5p}{p-1}} \, du \right)^{\frac{p-1}{5p}} \nonumber \\ 
	& \quad \leq C\nu(v)\Vert w_{q,\vartheta,\beta} f \Vert_{L^\infty}\left ( \int_{\R^3} \vert f(u) \vert ^{p'} \, du \right)^{1/p'}, \nonumber 
\end{align}
where $p'=\frac{5p}{p-1}$. Thus, we obtain \eqref{gamma_loss}. Next, we consider the gamma gain term $\Gamma_+(f,f)$. We notice 
\begin{equation*}
	\frac{1}{2} \vert v \vert ^2 \leq \vert v' \vert ^2 \quad \textrm{or} \quad \frac{1}{2} \vert v \vert ^2 \leq \vert u' \vert^2, 
\end{equation*}
which comes from $\vert v \vert^2 \leq \vert v' \vert^2 + \vert u' \vert^2$. Hence, one gets that 
\begin{align*}
	&\left \vert w_{q,\vartheta,\beta}(v) \Gamma_+ (f,f)(v) \right \vert \\ 
	&\quad \leq w_{q,\vartheta,\beta}(v) \int_{\R^3 \times \S^2} B(v-u,\omega) \sqrt{\mu(u)}  \left \vert f(u') f(v') \right \vert \textbf{1}_{\{\frac{1}{2}\vert v \vert^2 \leq \vert u' \vert^2\}} \, d\omega du \\
	&\quad \quad + w_{q,\vartheta,\beta}(v) \int_{\R^3 \times \S^2} B(v-u,\omega) \sqrt{\mu(u)}  \left \vert f(u') f(v') \right \vert \textbf{1}_{\{\frac{1}{2}\vert v \vert^2 \leq \vert v' \vert^2\}} \, d\omega du \\
	&\quad \leq C_{\beta} \int_{\R^3 \times \S^2} B(v-u,\omega) \sqrt{\mu(u)} \left \vert w_{q,\vartheta,\beta}(u')f(u') \exp\left(\frac{q}{8}(1+(1+t)^{-\vartheta})\vert v' \vert^2\right) f(v') \right \vert d\omega du \\
	&\quad \quad +C_{\beta}\int_{\R^3 \times \S^2} B(v-u,\omega) \sqrt{\mu(u)} \left \vert w_{q,\vartheta,\beta}(v')f(v') \exp\left(\frac{q}{8}(1+(1+t)^{-\vartheta})\vert u' \vert^2\right) f(u') \right \vert d\omega du \\
	&\quad := J_1+J_2.
\end{align*}
For $J_1$, as in \cite{Glassey}, we rewrite the variables as
\begin{equation*}
	V=u-v, \quad V_{\parallel} = \left(V\cdot \omega\right) \omega, \quad V_{\perp} = V-V_{\parallel}, \quad \eta = v+V_{\parallel},
\end{equation*}
which gives 
\begin{equation*}
v' = v+V_{\parallel} = \eta, \quad u' =v+V_{\perp}.
\end{equation*}
Hence, it holds that 
\begin{align*}
	J_1 &\leq C_{\beta} \Vert w_{q,\vartheta,\beta} f \Vert_{L^\infty} \int_{\R^3 \times \S^2} B(v-u,\omega) \sqrt{\mu(u)} \left \vert \exp\left(\frac{q}{8}(1+(1+t)^{-\vartheta})\vert v' \vert^2\right) f(v') \right \vert d\omega du\\
	&\leq C_{\beta} \Vert w_{q,\vartheta,\beta} f \Vert_{L^\infty} \left( \int_{\R^3} \vert v-u \vert^{p\gamma} \sqrt{\mu(u)} \, du\right)^{1/p} 
	\left(\int_{\R^3\times \S^2} \sqrt{\mu(u)}  \left \vert \exp\left(\frac{q}{8}(1+(1+t)^{-\vartheta})\vert v' \vert^2\right) f(v') \right \vert^{\frac{p}{p-1}}\, d\omega du\right)^{1-\frac{1}{p}}\\
	&\leq C_{\beta} \nu(v) \Vert w_{q,\vartheta,\beta} f \Vert_{L^\infty}\left( \int_{\R^3 \times \S^2} e^{-\frac{\vert v+V_{\parallel}+V_{\perp} \vert^2}{4}} \left \vert \exp\left(\frac{q}{8}(1+(1+t)^{-\vartheta})\left \vert v+V_{\parallel} \right \vert^2\right) f(v+V_{\parallel}) \right \vert^{\frac{p}{p-1}}\, d\omega dV\right)^{1-\frac{1}{p}}\\
	&\leq C_{\beta} \nu(v) \Vert w_{q,\vartheta,\beta} f \Vert_{L^\infty}\left( \int_{\R^3} \int_{V_{\perp}} \frac{1}{\vert V_{\parallel}\vert^2}e^{-\frac{\vert v+V_{\parallel}+V_{\perp} \vert^2}{4}}\left \vert \exp\left(\frac{q}{8}(1+(1+t)^{-\vartheta})\left \vert v+V_{\parallel} \right \vert^2\right) f(v+V_{\parallel}) \right \vert^{\frac{p}{p-1}}\, dV_{\perp} dV_{\parallel}\right)^{1-\frac{1}{p}}\\
	&= C_{\beta} \nu(v) \Vert w_{q,\vartheta,\beta} f \Vert_{L^\infty}\left( \int_{\R^3} \int_{V_{\perp}} \frac{1}{\vert \eta -v \vert^2}e^{-\frac{\vert \eta+V_{\perp} \vert^2}{4}} \left \vert \exp\left(\frac{q}{8}(1+(1+t)^{-\vartheta})\left \vert \eta \right \vert^2\right) f(\eta) \right \vert^{\frac{p}{p-1}}\, dV_{\perp} d\eta\right)^{1-\frac{1}{p}}\\
	&\leq C_{\beta} \nu(v) \Vert w_{q,\vartheta,\beta} f \Vert_{L^\infty} \left(\int_{\R^3} \frac{(1+\vert \eta \vert)^{-4}}{\vert \eta - v \vert^{5/2}}\, d\eta\right)^{\frac{4}{5}(1-\frac{1}{p})} \left(\int_{\R^3} (1+\vert \eta \vert)^{16} (1+ \vert \eta \vert^2)^{-\beta p'} \vert w_{q,\vartheta,\beta}f(\eta) \vert ^{p'} \, d\eta\right)^{1/p'}\\
	&\leq C_{\beta} \nu(v) \Vert w_{q,\vartheta,\beta} f \Vert_{L^\infty}\left(\int_{\R^3} (1+\vert \eta \vert)^{-2\beta p'+ 16} \vert w_{q,\vartheta,\beta}f(\eta) \vert ^{p'} \, d\eta\right)^{1/p'},
\end{align*}
where $p'=\frac{5p}{p-1}$. Notice that $J_2$ has a similar form to $J_1$ if the velocities $u'$ and $v'$ are interchanged. Hence, similar to the way of treating $J_1$, we get the same results in $J_2$ as well. Thus, the proof of Lemma \ref{gamma estimate} is complete. 
\end{proof}

\subsection{{Relative entropy}}
Recall the definition of the relative entropy \eqref{Relative entropy}
\begin{equation*}
\mathcal{E}(F) = \int_{\T^3 \times \R^3} \left(\frac{F}{\mu} \ln \frac{F}{\mu} - \frac{F}{\mu}+1 \right) \mu \, dxdv.
\end{equation*}
The below lemma is the global in time a priori estimate related with the relative entropy. 
\begin{lemma} \label{decay_entropy}
	Assume that $F$ satisfies the Boltzmann equation \eqref{def.be}. 
	\begin{equation} \notag
	\mathcal{E}(F) \leq \mathcal{E}(F_0), 
	\end{equation}
	for any $t\geq0$. 
\end{lemma}
\begin{proof}
	Define a function
	\begin{align*}
	h(s)=s\ln s -s +1,
	\end{align*} 
	for $s>0$. Then, $h$ is  a nonnegative and convex function on $(0,\infty)$ with 
	\begin{align*}
	h'(s)=\ln s. 
	\end{align*}
	From \eqref{def.be}, one can deduce that 
	\begin{align*}
	\partial_t [\mu h\left(\frac{F}{\mu}\right) ]+ \nabla_x \cdot [ v\mu h\left(\frac{F}{\mu}\right)] = Q(F,F) \ln \frac{F}{\mu}.
	\end{align*}
	Taking integration for $v\in \R^3$ and then for $x$ in $\T^3$, we get 
	\begin{align*}
	\frac{d}{dt} \int _{\T^3} \int_{\R^3} h\left(\frac{F}{\mu}\right) \mu \, dvdx  = \int _{\T^3} \int_{\R^3} Q(F,F) \ln F \, dvdx.   
	\end{align*}
	 From 
	\begin{align*}
	\int_{\T^3} \int _{\R^3} Q(F,F) \ln F \, dvdx \leq 0,
	\end{align*}
	we can deduce that 
	\begin{align*}
	\frac{d}{dt} \int_{\T^3} \int_{\R^3} h\left(\frac{F}{\mu}\right) \mu \, dvdx \leq 0.
	\end{align*}
	Furthermore, we have 
	\begin{equation*}
	\int_{\T^3} \int_{\R^3} h\left(\frac{F}{\mu}\right) \mu \, dvdx \leq \int_{\T^3} \int_{\R^3} h\left(\frac{F_0}{\mu}\right) \mu \, dvdx. 
	\end{equation*}
\end{proof}

\begin{lemma} \label{entropy_est}
	Assume that $F$ satisfies the Boltzmann equation \eqref{def.be}. Then, it holds that 
	\begin{equation} \notag
	\int_{\T^3 \times \R^3} \frac{1}{4\mu} \abs{F-\mu}^2  \textbf{1} _{\abs{F-\mu}\leq \mu} \,dxdv + \int_{\T^3 \times \R^3} \frac{1}{4} \abs{F-\mu} \textbf{1}_{\abs{F-\mu}>\mu} \, dxdv \leq \mathcal{E}(F_0), 
	\end{equation} 
	for any $t\geq0$. Furthermore, if we write $F$ in terms of the standard perturbation $f$, then 
	\begin{equation} \notag
	\int_{\T^3 \times \R^3} \frac{1}{4}\abs{f}^2  \textbf{1}_{\abs{f} \leq \sqrt{\mu}} \, dxdv + \int_{\T^3 \times \R^3} \frac{\sqrt{\mu}}{4} \abs{f} \textbf{1}_{\abs{f}>\sqrt{\mu}} \, dxdv \leq \mathcal{E}(F_0). 
	\end{equation} 
\end{lemma} 
\begin{proof}
	It is noticed that 
	\begin{align*}
	F\ln F - \mu \ln \mu = (1+\ln \mu) (F-\mu) + \frac{1}{2\tilde{F}} \abs{F-\mu}^2,
	\end{align*}
	where $\tilde{F}$ is between $F$ and $\mu$ form Taylor expansion. Then, we compute 
	\begin{align*}
	\frac{1}{2\tilde{F}} \abs{F-\mu}^2 &= F\ln F -\mu \ln \mu - (1+\ln \mu) (F-\mu) \\ 
	&= h\left(\frac{F}{\mu}\right) \mu.
	\end{align*}
	Thus, 
	\begin{align*}
	\int_{\T^3\times \R^3} \frac{1}{2\tilde{F}} \abs{F-\mu}^2 \, dvdx = \int_{\T^3\times \R^3} h\left (\frac{F}{\mu}\right ) \mu
	\, dvdx, 
	\end{align*}
	which is uniformly in time bounded in terms of Lemma 3.1. For the left-hand side, we write 
	\begin{align*}
	1= \textbf{1}_{\abs{F-\mu} \leq \mu} + \textbf{1}_{\abs{F-\mu} > \mu}.
	\end{align*}
	Over $\{ \abs {F-\mu} > \mu\}$, we have $F > 2\mu$ and hence 
	\begin{align*}
	\frac{\abs{F-\mu}}{\tilde{F}} = \frac{F-\mu}{\tilde{F}} \geq \frac{F-\frac{1}{2}F}{F} =\frac{1}{2}.
	\end{align*}
	Over $\{\abs{F-\mu} \leq \mu\}$, we have $0 \leq F \leq 2\mu$ and hence 
	\begin{align*}
	\frac{1}{\tilde {F}} \geq \frac{1}{2\mu}.
	\end{align*}
	Therefore, we obtain from Lemma \ref{decay_entropy} 
	\begin{align*}
	&\int_{\T^3\times \R^3}  \frac{1}{4\mu} \abs{F-\mu}^2 \textbf{1}_{\abs{F-\mu} \leq \mu} \,dvdx 
	+ \int_{\T^3\times \R^3} \frac{1}{4}\abs{F-\mu} \textbf{1}_{\abs{F-\mu}>\mu} \, dvdx \\
	& \leq \int_{\T^3\times \R^3} h\left(\frac{F}{\mu}\right) \mu \, dvdx \leq \int_{\T^3\times \R^3}  h\left(\frac{F_0}{\mu}\right) \mu \, dvdx =\mathcal{E}(F_0),
	\end{align*}
	for any $t\geq 0$. 
\end{proof}

\section{A priori estimate}
Let $F$ satisfy \eqref{def.be}, and denote that $F(t,x,v)=\mu(v)+\sqrt{\mu(v)}f(t,x,v)$.
Then we can rewrite the Boltzmann equation \eqref{def.be} for $h=w_{q,\vartheta,\beta}f$ as
\begin{align}\label{def.be.h}
\p_th(t,x,v)+v\cdot\nabla_xh(t,x,v) + \tilde{\nu}(v,t)h(t,x,v) = K_wh(t) + w\Gamma(f,f)(t),
\end{align}
where
\begin{align*}
\tilde{\nu}(v,t) := \nu(v) + \frac{\vartheta q|v|^2}{8(1+t)^{\vartheta+1}} .
\end{align*}
Let us fix $T > 0$ and assume a priori bound
\begin{align} \label{apriori}
	\sup_{0\le t \le T}\|h(t)\|_{L^{\infty}}\le \bar{M}.
\end{align}

To get the sub-exponential time-decay property of the linearized solution operator $G_v(t,s)$, we consider the linearized equation as follow : 
	\begin{align}\label{linear equation}
		\p_th+v\cdot \nabla_xh+ \tilde{\nu}h = 0,
	\end{align}
where $\tilde{\nu}$ is defined in \eqref{nutilde}. Then the solution of \eqref{linear equation} can be written by 
\begin{align*}
h(t,x,v)
&= e^{-\int_0^t \tilde{\nu}(v,\tau) d\tau}h_0(x-tv,v).\\
\end{align*}
We can define the solution operator $G_v(t,s)$ as follow : 
\begin{align*}
G_v(t,s) : = e^{-\int_s^t \tilde{\nu}(v,\tau) d\tau}.
\end{align*}
\begin{proposition}\label{prop_apriori}
 Let $h(t,x,v)$ satisfy the equation \eqref{def.be.h} and  $\rho-1 = \frac{(1+\vartheta)\gamma}{2-\gamma}$. Assume that \eqref{weight} holds. Let $0<t\le T<\infty$. Then it holds that
\begin{equation}
	\|h(t)\|_{L^{\infty}} \le Ce^{-\frac{\lambda}{2}(1+t)^{\rho}}\|h_0\|_{L^{\infty}}\left(\int_0^t\|h(s)\|_{L^{\infty}}ds +1\right) + D,
\end{equation}
where $0<\delta\ll1, 0<\epsilon\ll1$, and $N\gg1$ can be chosen arbitrary small and large and,
\begin{align*}
D := C_{\bar{M}}\epsilon^{\gamma+3} +C_{\bar{M}}\left(N^{\gamma}+\frac{1}{N^{\frac{3+\gamma}{2}}}+\frac{1}{N+1}\right)+\frac{C_{\bar{M},\epsilon}}{N^2} +C_{\bar{M},N}\delta^{\frac{1}{p'}} + C_{\bar{M},N,\epsilon,\delta}\left(\mathcal{E}(F_0)^{\frac{1}{2}}+\mathcal{E}(F_0)^{\frac{1}{2p'}}\right).
\end{align*}
\end{proposition}

\begin{proof}
Take $(t,x,v) \in (0,T]\times\mathbb{T}^3\times\mathbb{R}^3$. By Duhamel's principle, it holds that
\begin{align*}
|h(t,x,v)| 
&\le |G_v(t,0)h_0(x-tv,v)|\\
&\indent + \int_0^t|G_v(t,s)|[|K_wh(s,x-(t-s)v,v)| + |w\Gamma_+(f,f)(s,x-(t-s)v,v)| + |w\Gamma_-(f,f)(s,x-(t-s)v,v)|] ds\\
&= I_1 + I_2 + I_3+ I_4.
\end{align*}
First of all, note that $G_v(t,s)\le e^{-\nu(v)(t-s)}$ and we can get
\begin{align*}
\int_0^t G_v(t,s)\nu(v) ds \le \int_0^t e^{-\nu(v)(t-s)}\nu(v) ds =1-e^{-\nu(v)t} \le 1.
\end{align*}
By extracting the collision frequency $\nu$ from the other terms, we will deal with the time terms $G_{v}(t,s)$ and $G_u(s,s')$ as above except for the terms including $h_0$. For \eqref{nutilde}, notice that $(1+|v|)\textbf{1}_{|v|\le 1} \le 2$ and $\textbf{1}_{|v|\ge 1}\le |v|^2$. Then it follows that
\begin{align*}
1+|v|^2 + \nu(v)
&= \textbf{1}_{|v|\le 1} + \textbf{1}_{|v|\ge 1} + |v|^2 + \nu(v)\\
&\le 2^{-\gamma}(1+|v|)^{\gamma} + |v|^2 + |v|^2 + \nu(v)\\
&\le C(|v|^2 + \nu(v)).
\end{align*}
Denote $\displaystyle a=\frac{\gamma-2}{\gamma}$ and $\displaystyle b=\frac{2-\gamma}{2}$. Then $\frac{1}{a} + \frac{1}{b}=1$ and $a>1, b>1$. Using the above inequality and  Young's inequality, we gain
\begin{align}\label{nutilde estimate}
\begin{split}
\tilde{\nu}(v,t)
&= \frac{\vartheta q|v|^2}{8(1+t)^{\vartheta+1}} + \nu(v)\\
&\ge C\left\{(1+t)^{-(1+\vartheta)}[|v|^2+\nu(v)]+(1-(1+t)^{-(1+\vartheta)})\nu(v))\right\}\\
&\ge C\left\{(1+t)^{-(1+\vartheta)}[1+|v|^2+\nu(v)]+(1-(1+t)^{-(1+\vartheta)})\nu(v)\right\}\\
&\ge C\left\{(1+t)^{-\frac{a(1+\vartheta)}{a}}\nu(v)^{\frac{2a}{a\gamma}}+\nu(v)^{\frac{b}{b}}\right\}\\
&\ge C\left\{(1+t)^{-\frac{1+\vartheta}{a}}\nu(v)^{\frac{2}{a\gamma}}\nu(v)^{\frac{1}{b}}\right\} = C(1+t)^{\frac{(1+\vartheta)\gamma}{2-\gamma}}.
\end{split}
\end{align}
For $I_1$, by the $\tilde{\nu}$ estimate \eqref{nutilde estimate}, we obtain
\begin{align}\label{I_1 estimate}
\begin{split}
I_1
&\le e^{-\int_0^t C(1+\tau)^{\rho-1} d\tau}\|h_0\|_{L^{\infty}} \le Ce^{-\lambda(1+t)^{\rho}}\|h_0\|_{L^{\infty}},
\end{split}
\end{align}
where $\lambda=\frac{C}{\rho}>0$. For $I_2$, we divde $I_2$ into four cases where i) $|v| \ge N$, ii) $|v|\le N$ and $|u| \ge 2N$, iii) $|v|\le N$, $|u|\le 2N$, and $|u^*|\ge3N$, and iv) $|v|\le N$, $|u|\le 2N$, and $|u^*|\le3N$.\\
\\
\indent \textbf{(Case 1)} $|v|\ge N$\\
\\
See that $\nu(v) \sim (1+v)^{\gamma}$ and $\nu(v)^{-1}\mu(v)^{\frac{1-q}{16}}$ is bounded for $|v|\ge N$. By Lemma \ref{Lemma_k.1}, $I_2\textbf{1}_{\{|v|\ge N\}}$ is bounded by
\begin{align}\label{I_2 case1}
\begin{split}
 I_2\textbf{1}_{\{|v|\ge N\}}
&\le \int_0^t e^{-\nu(v)(t-s)}\textbf{1}_{\{|v|\ge N\}}\left[\left|K_w^{1-\chi}h(s)\right| + \left|K_w^{\chi}h(s)\right|\right] ds\\
&\le \int_0^t e^{-\nu(v)(t-s)}\nu(v)\textbf{1}_{\{|v|\ge N\}}\left[C\nu^{-1}(v)\mu^{\frac{1-q}{8}}(v)\epsilon^{\gamma+3} + C_{\epsilon}\nu^{-1}(v)\langle v \rangle^{\gamma-2}\right]\|h(s)\|_{L^{\infty}} ds\\
&\le \int_0^t e^{-\nu(v)(t-s)}\nu(v)\left[C\mu^{\frac{1-q}{16}}(v)\epsilon^{\gamma+3} + \frac{C_{\epsilon}}{1+N^2}\right]\|h(s)\|_{L^{\infty}} ds\\
&\le \left(C\epsilon^{\gamma+3} + \frac{C_{\epsilon}}{N^2}\right)\sup_{0\le s \le T}\|h(s)\|_{L^{\infty}}.\\
\end{split}
\end{align}
\indent \textbf{(Case 2)} $|v| \le N$ and $|u| \ge 2N$\\
\\
Observe that $|u-v|\ge N$ and \eqref{e.4} holds although there is no $\frac{1}{|\eta|}$ in \eqref{e.4}. It follows from Lemma \ref{Lemma_k.2}, \eqref{e.1}, \eqref{e.3}, and \eqref{e.4} that
\begin{align}\label{case2 k_w,2}
\begin{split}
\int_{|u|\ge 2N} k_{w,2}(v,u)\textbf{1}_{\{|v|\le N\}} du 
&\le \int_{|u|\ge2N}\frac{C_{\gamma}}{|v-u|^{\frac{3-\gamma}{2}}}\textbf{1}_{\{|v|\le N\}}e^{-\frac{|v-u|^2}{8}-\frac{||u|^2-|v|^2|^2}{8|u-v|^2}}\frac{w_{q,\vartheta,\beta}(v)}{w_{q,\vartheta,\beta}(u)}du\\
&\le \frac{C_{\gamma}}{N^{\frac{3-\gamma}{2}}}\int_{|u|\ge2N}e^{-\frac{|v-u|^2}{8}-\frac{||u|^2-|v|^2|^2}{8|u-v|^2}}\frac{w_{q,\vartheta,\beta}(v)}{w_{q,\vartheta,\beta}(u)}du\\
&\le \frac{C_{\gamma}}{N^{\frac{3-\gamma}{2}}(1+|v|)}.
\end{split}
\end{align}
Note that $\nu(v)^{-1}(1+|v|^2)^{\beta}e^{-\frac{1-q}{4}|v|^2}$ is bounded for $|v|\le N$. Then we can gain by Lemma \ref{Lemma_k.2} and \eqref{case2 k_w,2}
\begin{align}\label{I_2 case2}
\begin{split}
I_2\textbf{1}_{\{|v|\le N, |u| \ge 2N\}}
&\le \int_0^t e^{-\nu(v)(t-s)}\textbf{1}_{\{|v|\le N\}}\int_{|u|\ge 2N}\left(|k_{w,1}(v,u)h(s)| + |k_{w,2}(v,u)h(s)|\right) du ds \\
&\le \int_0^t e^{-\nu(v)(t-s)}\|h(s)\|_{L^{\infty}}\textbf{1}_{\{|v|\le N\}}\\
&\indent \times\int_{|u|\ge 2N}C(1+|v|^2)^{\beta}|v-u|^{\gamma}e^{-\frac{1-q}{4}|v|^2}e^{-\frac{|u|^2}{4}} + \frac{C_{\gamma}}{|v-u|^{\frac{3-\gamma}{2}}}e^{-\frac{|v-u|^2}{8}-\frac{||u|^2-|v|^2|^2}{8|u-v|^2}}\frac{w_{q,\vartheta,\beta}(v)}{w_{q,\vartheta,\beta}(u)}du ds\\
&\le \int_0^t e^{-\nu(v)(t-s)}\nu(v)\|h(s)\|_{L^{\infty}}\textbf{1}_{\{|v|\le N\}}\\
&\indent \times\left(\int_{|u|\ge 2N}CN^{\gamma}\nu(v)^{-1}(1+|v|^2)^{\beta}e^{-\frac{1-q}{4}|v|^2}e^{-\frac{|u|^2}{4}}du + \frac{C_{\gamma}\nu(v)^{-1}}{N^{\frac{3-\gamma}{2}}(1+|v|)}\right)ds\\
&\le \int_0^t e^{-\nu(v)(t-s)}\nu(v)\|h(s)\|_{L^{\infty}} \left(CN^{\gamma}\int_{|u|\ge 2N}e^{-\frac{|u|^2}{4}} du + \frac{C_{\gamma}}{(1+|v|)N^{\frac{3+\gamma}{2}}} \right) ds\\
&\le \left(CN^{\gamma} + \frac{C_{\gamma}}{N^{\frac{3+\gamma}{2}}}\right)\sup_{0\le t \le T}\|h(s)\|_{L^{\infty}},
\end{split}
\end{align}
where $\gamma <0$ and $\frac{\gamma+3}{2}>0$.\\
\\
\indent \textbf{(Case 3)} $|v|\le N$, $|u| \le 2N$, and $|u^*|\ge 3N$,\\
\\
Denote $h(s')=h(s',x-(t-s)v-(s-s')u,u^*)$. We can split $I_2\textbf{1}_{\{|v|\le N, |u| \le 2N, |u^*|\ge 3N \}}$ by Lemma \ref{gamma estimate} into
\begin{align*}
I_2\textbf{1}_{\{|v|\le N, |u| \le 2N, |u^*|\ge 3N \}}
&= \int_0^t \int_{|u|\le 2N} \int_0^s G_v(t,s)G_u(s,s')k_w(v,u)\textbf{1}_{\{|v|\le N\}}\\
&\indent \times \left[|K_wh(s')| + |w\Gamma_+(f,f)(s')| + |w\Gamma_-(f,f)(s')|\right] ds'duds\\
&\le J_1 + J_2 + J_3,
\end{align*}
where
\begin{align*}
J_1 
&:= \int_0^t \int_{|u|\le 2N} \int_0^s \int_{|u*|\ge 3N}e^{-\nu(v)(t-s)}e^{-\nu(u)(s-s')}\textbf{1}_{\{|v|\le N\}}k_w(v,u)k_w(u,u^*)|h(s')| du^*ds'duds,\\
J_2
&:= C\int_0^t \int_{|u|\le 2N} \int_0^s e^{-\nu(v)(t-s)}e^{-\nu(u)(s-s')}\nu(u)k_w(v,u)\textbf{1}_{\{|v|\le N\}}\|h(s')\|_{L^{\infty}}\\
&\indent \times\left(\int_{|u^*|\ge 3N} (1+|u^*|)^{-2\beta p' + 16}|h(u^*)|^{p'}du^*\right)^{\frac{1}{p'}}ds'duds,\\
J_3
&:= C\int_0^t \int_{|u|\le 2N} \int_0^s e^{-\nu(v)(t-s)}e^{-\nu(u)(s-s')}\nu(u)k_w(v,u)\textbf{1}_{\{|v|\le N\}}\|h(s')\|_{L^{\infty}}\\
&\indent \times\left(\int_{|u^*|\ge 3N} \frac{1}{w_{q,\vartheta,\beta}(u^*)^{p'}}|h(u^*)|^{p'}du^*\right)^{1/p'}ds'duds.
\end{align*}
For $J_1$, notice that $|u^*-u|\ge N$, $|u-v|^{-\gamma} \le (1+|u-v|)^{-\gamma}$ and \eqref{e.3} holds although $(1+|\eta|)^{\beta}$ is replaced to $(1+|\eta|)^{\beta-\gamma}$. Then we can get by Lemma \ref{Lemma_k.2}, \eqref{e.1}, \eqref{e.3}, and \eqref{e.4}
\begin{align}\label{case3 k_w,2}
\begin{split}
\int_{|u^*|\ge 3N} k_{w,2}(u,u^*) \textbf{1}_{\{|u|\le 2N\}}du^*
&\le \int_{|u^*|\ge 3N} \textbf{1}_{\{|u|\le 2N\}}\frac{C_{\gamma}|u^*-u|^{-\gamma}}{|u^*-u|^{\frac{3-3\gamma}{2}}}e^{-\frac{|u^*-u|^2}{8}-\frac{||u|^2-|u^*|^2|^2}{8|u-u^*|^2}}\frac{w_{q,\vartheta,\beta}(u)}{w_{q,\vartheta,\beta}(u^*)} du^*\\
&\le \frac{C_{\gamma}}{N^{\frac{3-3\gamma}{2}}}\int_{|u*|\ge 3N}|u^*-u|^{-\gamma}e^{-\frac{|u^*-u|^2}{8}-\frac{||u|^2-|u^*|^2|^2}{8|u-u^*|^2}}\frac{w_{q,\vartheta,\beta}(u)}{w_{q,\vartheta,\beta}(u^*)}du^*\\
&\le \frac{C_{\gamma}}{N^{\frac{3-3\gamma}{2}}(1+|u|)}.
\end{split}
\end{align}
Recall that $k_w(v,u)$ is integrable on $\{|u|\le 2N\}$. From \eqref{case2 k_w,2} and Lemma \ref{Lemma_k.2}, $J_1$ is bounded similarly to the estimate of $I_2$ for $|v|\le N$ and $|u|\ge 2N$ by
\begin{align}\label{J_1 estimate}
\begin{split}
J_1
&= \int_0^t \int_{|u|\le 2N} \int_0^s \int_{|u*|\ge 3N}e^{-\nu(v)(t-s)}e^{-\nu(u)(s-s')}\textbf{1}_{\{|v|\le N\}}k_w(v,u)k_w(u,u^*)|h(s')| du^*ds'duds\\
&\le \int_0^t \int_{|u|\le 2N} \int_0^s e^{-\nu(v)(t-s)}e^{-\nu(u)(s-s')}\textbf{1}_{\{|v|\le N\}}k_w(v,u)\int_{|u*|\ge 3N}(k_{w,1}(u,u^*)+k_{w,2}(u,u^*))|h(s')| du^*ds'duds \\
&\le \int_0^t \int_{|u|\le 2N} \int_0^s e^{-\nu(v)(t-s)}e^{-\nu(u)(s-s')}\textbf{1}_{\{|v|\le N\}}k_w(v,u)\\
&\indent \times\left(\int_{|u*|\ge 3N}C(1+|u|^2)^{\beta}|u^*-u|^{\gamma}e^{-\frac{1-q}{4}|u|^2}e^{-\frac{|u^*|^2}{4}}|h(s')|du^*+\frac{C_{\gamma}}{N^{\frac{3-3\gamma}{2}}(1+|u|)}\sup_{0\le s' \le T}\|h(s')\|_{L^{\infty}}\right) ds'duds \\
&\le C\int_0^t \int_{|u|\le 2N} \int_0^s \nu(v)e^{-\nu(v)(t-s)}\nu(u)e^{-\nu(u)(s-s')}\textbf{1}_{\{|v|\le N\}}k_w(v,u)\\
&\indent \times\nu(u)^{-1}\nu(v)^{-1}\left(N^{\gamma}\nu(u)e^{-\frac{9}{8}N^2}\int_{|u*|\ge 3N}e^{-\frac{|u^*|^2}{8}}|h(s')|du^* + \frac{C_{\gamma}}{N^{\frac{3-3\gamma}{2}}}\sup_{0\le s' \le T}\|h(s')\|_{L^{\infty}}\right)ds'duds \\
&\le C\int_0^t \int_{|u|\le 2N} \int_0^s \nu(v)e^{-\nu(v)(t-s)}\nu(u)e^{-\nu(u)(s-s')}k_w(v,u)\\
&\indent \times\left\{\left(1+\frac{1}{N}\right)^{\gamma}e^{-\frac{9}{8}N^2} + \frac{\{(1+N)(1+2N)\}^{-\gamma}}{N^{\frac{3-3\gamma}{2}}}\right\}\sup_{0\leq s' \leq T} \|h(s')\|_{L^{\infty}}ds'duds\\
&\le C\int_0^t \int_{|u|\le 2N} \nu(v)e^{-\nu(v)(t-s)}k_w(v,u)\left(e^{-\frac{9}{8}N^2} + \frac{1}{N^{\frac{3+\gamma}{2}}}\right)\sup_{0\le s' \leq T}\|h(s')\|_{L^{\infty}}duds\\
&\le C\left(e^{-\frac{9}{8}N^2} + \frac{1}{N^{\frac{3+\gamma}{2}}}\right)\sup_{0\le s' \le T}\|h(s')\|_{L^{\infty}},\\
\end{split}
\end{align}
because $(1+|u|^2)^{\beta}e^{-\frac{1-q}{4}|u|^2} \le C\nu(u)$ for $|u|\le 2N$. For $J_2$, denote that $h(u^*) = h(s',x-(t-s)v-(s-s')u,u^*)$ and we use Lemma \ref{gamma estimate} to get
\begin{align}\label{J_2 estimate}
\begin{split}
J_2
&= C\int_0^t \int_{|u|\le 2N} \int_0^s e^{-\nu(v)(t-s)}e^{-\nu(u)(s-s')}\nu(u)k_w(v,u)\textbf{1}_{\{|v|\le N\}}\|h(s')\|_{L^{\infty}}\\
&\indent \times\left(\int_{|u^*|\ge 3N} (1+|u^*|)^{-2\beta p' + 16}|h(u^*)|^{p'}du^*\right)^{\frac{1}{p'}}ds'duds\\
&\le C\int_0^t \int_{|u|\le 2N} \int_0^s \nu(v)e^{-\nu(v)(t-s)}e^{-\nu(u)(s-s')}\nu(u)k_w(v,u)\textbf{1}_{\{|v|\le N\}}\|h(s')\|_{L^{\infty}}\\
&\indent \times\frac{\nu^{-1}(v)}{(1+3N)^3}\left(\int_{|u^*|\ge 3N} (1+|u^*|)^{-(2\beta-3)p' + 16}|h(u^*)|^{p'}du^*\right)^{1/p'}ds'duds\\
&\le C\int_0^t \int_{|u|\le 2N} \int_0^s \nu(v)e^{-\nu(v)(t-s)}e^{-\nu(u)(s-s')}\nu(u)k_w(v,u)\|h(s')\|^2_{L^{\infty}}\\
&\indent \times\frac{(1+N)^{-\gamma}}{(1+3N)^3}\left(\int_{|u^*|\ge 3N} (1+|u^*|)^{-(2\beta-3)p' + 16} du^*\right)^{1/p'}ds'duds\\
&\le \frac{C}{(1+3N)^{3+\gamma}}\int_0^t \int_{|u|\le 2N} \int_0^s \nu(v)e^{-\nu(v)(t-s)}e^{-\nu(u)(s-s')}\nu(u)k_w(v,u)\|h(s')\|^2_{L^{\infty}}ds'duds\\
&\le \frac{C}{(1+3N)^{3+\gamma}}\sup_{0\le s' \le T}\|h(s')\|^2_{L^{\infty}},
\end{split}
\end{align}
since $p'\ge 5$ and $\beta \ge 7/2$ yield  $-(2\beta-3)p' + 16 < -3$ and $k_w(u,v)$ is integrable on $\{|u|\le 2N\}$. For $J_3$, denote that $h(u^*) = h(s',x-(t-s)v-(s-s')u,u^*)$ and we can get
\begin{align*}
J_3
&= C\int_0^t \int_{|u|\le 2N} \int_0^s e^{-\nu(v)(t-s)}e^{-\nu(u)(s-s')}\nu(u)k_w(u,v)\textbf{1}_{\{|v|\le N\}}\|h(s')\|_{L^{\infty}}\\
&\indent \times\left(\int_{|u^*|\ge 3N} \frac{1}{w_{q,\vartheta,\beta}(u^*)^{p'}}|h(u^*)|^{p'}du^*\right)^{1/p'}ds'duds\\
&\le C\int_0^t \int_{|u|\le 2N} \int_0^s e^{-\nu(v)(t-s)}e^{-\nu(u)(s-s')}\nu(u)k_w(u,v)\textbf{1}_{\{|v|\le N\}}\|h(s')\|_{L^{\infty}}\\
&\indent \times\left(\int_{|u^*|\ge 3N} (1+|u^*|)^{-2\beta p' + 16}|h(u^*)|^{p'}du^*\right)^{1/p'}ds'duds.\\
\end{align*}
By the above estimate, we bound $J_3$ similarly to an estimate of $J_2$ by
\begin{align}\label{J_3 estimate}
J_3\le \frac{C}{(1+3N)^{3+\gamma}}\sup_{0\le s' \le T}\|h(s')\|^2_{L^{\infty}}.
\end{align}
Combining \eqref{J_1 estimate}, \eqref{J_2 estimate}, and \eqref{J_3 estimate} altogether, we obtain
\begin{align}\label{I_2 case3}
I_2\textbf{1}_{\{|v|\le N, |u| \le 2N, |u^*|\ge 3N \}}
&\le C_{\bar{M}}\left(e^{-\frac{9}{8}N^2}+\frac{1}{N^{\frac{3+\gamma}{2}}}+\frac{1}{(1+3N)^{3+\gamma}}\right).
\end{align}
\indent \textbf{(Case 4)} $|v|\le N$, $|u|\le 2N$, and $|u^*|\le 3N$\\
\\
By Duhamel's principle and Lemma \ref{gamma estimate}, we can separate $I_2\textbf{1}_{\{|v|\le N, |u| \le 2N, |u^*|\le 3N\}}$ into
\begin{align*}
I_2\textbf{1}_{\{|v|\le N, |u| \le 2N, |u^*|\le 3N\}}
&= \int_0^t  G_v(t,s)\textbf{1}_{\{|v|\le N\}}|K_wh(s)| ds\\
&\le \int_0^t G_v(t,s)\textbf{1}_{\{|v|\le N\}}[|K_w^{1-\chi}h(s)| + |K_w^{\chi}h(s)|]ds\\
&\le I_{20}+I_{21}+I_{22}+I_{23} + I_{24}+I_{25},
\end{align*}
where
\begin{align*}
I_{20}
&:=\int_0^t e^{-\nu(v)(t-s)}\textbf{1}_{\{|v|\le N\}}|K_w^{1-\chi}h(s)|ds,\\
I_{21}
&:=\int_0^t \int_{|u|\le 2N} G_v(t,s)G_u(s,0)\textbf{1}_{\{|v|\le N\}}k_w^{\chi}(v,u)|h_0| duds,\\
I_{22}
&:=\int_0^t \int_{|u|\le 2N} \int_0^{s-\delta}\int_{|u^*|\le 3N}e^{-\nu(v)(t-s)}e^{-\nu(u)(s-s')}\textbf{1}_{\{|v|\le N\}}k_w^{\chi}(v,u)k_w(u,u^*)|h(s')| du^*ds'duds,\\
I_{23}
&:=C\int_0^t \int_{|u|\le 2N}\int_0^{s-\delta} e^{-\nu(v)(t-s)}\textbf{1}_{\{|v|\le N\}}k_w^{\chi}(v,u)\nu(u)e^{-\nu(u)(s-s')}\|h(s')\|_{L^{\infty}}\\
&\indent \times\left(\int_{|u^*|\le 3N}(1+|u^*|)^{-2\beta p'+16}|h(u^*)|^{p'}du^*\right)^{1/p'}ds'duds,\\
I_{24}
&:= C\int_0^t \int_{|u|\le 2N} \int_0^{s-\delta}e^{-\nu(v)(t-s)}e^{-\nu(u)(s-s')}\textbf{1}_{\{|v|\le N\}}k_w^{\chi}(v,u)\nu(u)\|h(s')\|_{L^{\infty}}\\
&\indent \times\left(\int_{|u^*|\le 3N}w_{q,\vartheta,\beta}(u^*)^{-p'}|h(u^*)|^{p'}du^*\right)^{1/p'}ds'duds,\\
I_{25}
&:=\int_0^t \int_{|u|\le 2N} \int_{s-\delta}^s e^{-\nu(v)(t-s)}e^{-\nu(u)(s-s')}\textbf{1}_{\{|v|\le N\}}k_w^{\chi}(v,u)\\
&\indent \times\left(|K_wh(s')| +|w\Gamma_+(f,f)(s')|+|w\Gamma_-(f,f)(s')|\right) ds'duds.
\end{align*}
For $I_{20}$, it follows from Lemma \ref{Lemma_k.1} that
\begin{align}\label{I_20 estimate}
\begin{split}
I_{20}
&=\int_0^t  e^{-\nu(v)(t-s)}\textbf{1}_{\{|v|\le N\}}|K_w^{1-\chi}h(s)|ds\\
&\le C\int_0^t e^{-\nu(v)(t-s)}\mu(v)^{\frac{1-q}{8}}\epsilon^{\gamma+3}\|h(s)\|_{L^{\infty}} ds\\
&\le C\epsilon^{\gamma+3}\sup_{0 \le s \le T}\|h(s)\|_{L^{\infty}}.
\end{split}
\end{align}
For $I_{21}$, recollect that $k_w^{\chi}(v,u) \le k_w(v,u)$ and $k_w(v,u)$ is integrable on $\{|u|\le 2N\}$. Then, by \eqref{nutilde estimate}, we obtain
\begin{align}\label{I_21 estimate}
\begin{split}
I_{21}
&=\int_0^t \int_{|u|\le 2N} G_v(t,s)G_u(s,0)\textbf{1}_{\{|v|\le N\}}k_w^{\chi}(v,u)|h_0| duds\\
&\le \int_0^t \int_{|u|\le 2N} G_v(t,s)G_u(s,0)k_w(v,u)|h_0| duds\\
&\le \int_0^t \int_{|u|\le 2N} e^{-\int_s^t \tilde{\nu}(v,\tau) d\tau}e^{-\int_0^s \tilde{\nu}(u,\tau) d\tau}k_w(v,u)|h_0| duds\\
&\le C\int_0^t \int_{|u|\le 2N} e^{-\lambda\{(1+t)^{\rho}-(1+s)^{\rho}\}}e^{-\lambda(1+s)^{\rho}}k_w(v,u)|h_0| duds\\
&\le C\int_0^t \int_{|u|\le 2N} e^{-\lambda(1+t)^{\rho}}k_w(v,u)|h_0| duds\\
&\le Ce^{-\frac{\lambda}{2}(1+t)^{\rho}}\|h_0\|_{L^{\infty}},
\end{split}
\end{align}
because $te^{-\frac{\lambda}{2}(1+t)^{\rho}}$ is bounded for $t\ge0$. For $I_{22}$, we can divide $I_{22}$ into
\begin{align*}
I_{22}
&=\int_0^t \int_{|u|\le 2N} \int_0^{s-\delta}\int_{|u^*|\le 3N}e^{-\nu(v)(t-s)}e^{-\nu(u)(s-s')}\textbf{1}_{\{|v|\le N\}}k_w^{\chi}(v,u)k_w(u,u^*)|h(s')| du^*ds'duds\\
&= \int_0^t \int_{|u|\le 2N} \int_0^{s-\delta}\int_{|u^*|\le 3N}e^{-\nu(v)(t-s)}e^{-\nu(u)(s-s')}\textbf{1}_{\{|v|\le N\}}k_w^{\chi}(v,u)k_w^{1-\chi}(u,u^*)|h(s')| du^*ds'duds\\
&\indent + \int_0^t \int_{|u|\le 2N} \int_0^{s-\delta}\int_{|u^*|\le 3N} e^{-\nu(v)(t-s)}e^{-\nu(u)(s-s')}\textbf{1}_{\{|v|\le N\}}k_w^{\chi}(v,u)k_w^{\chi}(u,u^*)|h(s')| du^*ds'duds\\
&=: I_{221} + I_{222}.
\end{align*}
For $I_{221}$, recall that $k_w^{\chi}(v,u) \le k_w(v,u)$ and $k_w(v,u)$ is integrable on $\{|u|\le 2N\}$. Then by Lemma \ref{Lemma_k.1} we have
\begin{align}\label{I_221 estimate}
\begin{split}
I_{221}
&= \int_0^t \int_{|u|\le 2N}\int_0^{s-\delta} e^{-\nu(v)(t-s)}e^{-\nu(u)(s-s')}\textbf{1}_{\{|v|\le N\}}k_w^{\chi}(v,u)\int_{|u^*|\le 3N}k_w^{1-\chi}(u,u^*)|h(s')|du^*ds'duds\\
&\le C\int_0^t \int_{|u|\le 2N}\int_0^{s-\delta} e^{-\nu(v)(t-s)}e^{-\nu(u)(s-s')}k_w^{\chi}(v,u) \mu(u)^{\frac{1-q}{8}}\epsilon^{\gamma+3}\Vert h(s')\Vert_{L^\infty} ds'duds\\
&= C\epsilon^{\gamma+3}\int_0^t \int_{|u|\le 2N}\int_0^{s-\delta} e^{-\nu(v)(t-s)}e^{-\nu(u)(s-s')}\\
&\indent \times k^{\chi}(v,u)\frac{(1+|v|^2)^{\beta+3}}{(1+|u|^2)^{\beta+3}}\frac{(1+|u|^2)^{3}}{(1+|v|^2)^{3}} \frac{w_{q,\vartheta,0}(v)}{w_{q,\vartheta,0}(u)}\mu(u)^{\frac{1-q}{8}}\Vert h(s')\Vert_{L^\infty} ds'duds\\
&\le C\epsilon^{\gamma+3}\int_0^t \int_{|u|\le 2N}\int_0^{s-\delta} \nu(v)e^{-\nu(v)(t-s)}\nu(u)e^{-\nu(u)(s-s')}k^{\chi}_{w_{q,\vartheta,\beta+3}}(v,u)\Vert h(s') \Vert_{L^\infty} ds'duds\\
&\le C \epsilon^{\gamma+3}\int_0^t \int_{|u|\le 2N}\int_0^{s-\delta} \nu(v)e^{-\nu(v)(t-s)}\nu(u)e^{-\nu(u)(s-s')}k_{w_{q,\vartheta,\beta+3}}(v,u)\Vert h(s') \Vert_{L^\infty}  ds'duds\\
&\le C \epsilon^{\gamma+3}\sup_{0\le s' \le T}\|h(s')\|_{L^{\infty}},
\end{split}
\end{align}
since $(1+|u|^2)^3\nu(u)^{-1}\mu(u)^{\frac{1-q}{8}}$ is bounded for $|u|\le 2N$ and $(1+|v|^2)^{-3} \le C\nu(v)$ for $|v|\le N$. For $I_{222}$, we can get
\begin{align*}
I_{222}
&\le \int_0^t \int_{|u|\le 2N}\int_0^{s-\delta} e^{-\nu(v)(t-s)}\nu(u)^{-1}\nu(u)e^{-\nu(u)(s-s')}\textbf{1}_{\{|v|\le N\}}\int_{|u^*|\le 3N}k_w^{\chi}(v,u)k_w^{\chi}(u,u^*)|h(s')|du^*ds'duds\\
&\le (1+2N)^{-\gamma}\int_0^t e^{-\nu(v)(t-s)}\textbf{1}_{\{|v|\le N\}} \sup_{0\le s' \le s-\delta}\left(\int_{|u|\le 2N}\int_{|u^*|\le 3N}k_w^{\chi}(v,u)k_w^{\chi}(u,u^*)|h(s')|du^*du\right)ds\\
&\le (1+2N)^{-\gamma}\int_0^t e^{-\nu(v)(t-s)}\textbf{1}_{\{|v|\le N\}}\\
&\indent \times\sup_{0\le s' \le s-\delta}\|k_w^{\chi}(v,u)k_w^{\chi}(u,u^*)\|_{L^2_{u,u^*}}\left(\int_{|u|\le 2N}\int_{|u^*|\le 3N}|h(s',x-(t-s)v-(s-s')u,u^*)|^2 du^*du\right)^{1/2} ds,\\
\end{align*}
in that $k_w^{\chi}(v,u)k_w^{\chi}(u,u^*) \in L^2(\{|u|\le 2N\}\times \{|u^*|\le 3N\})$. We change the variables as follow and denote $x'=x-(t-s)v$. Then we have
\begin{align}\label{change of variable}
\begin{split}
y=x-(t-s)v-(s-s')u=x'-(s-s')u,\indent\left|\det \frac{\p(y,u^*)}{\p(u,u^*)} \right| = \frac{1}{|s-s'|^3}.
\end{split}
\end{align}
Besides, $\|h(s')\|_{L^2_xL^2_v\left(\{\abs{u*}\leq 3N\}\right)}$ can be bounded from Lemma \ref{entropy_est} by 
\begin{align}\label{mathcalE}
\begin{split}
&\int_{y\in \T^3}\int_{|u^*|\le 3N} |h(s',y,u^*)|^2 du^*dy\\
&\le \int_{y\in \T^3}\int_{|u^*|\le 3N} |h(s',y,u^*)|^2\textbf{1}_{|f|\le \sqrt{\mu}} du^*dy + \int_{y\in \T^3}\int_{|u^*|\le 3N} |h(s',y,u^*)|^2\textbf{1}_{|f|\ge \sqrt{\mu}} du^*dy\\
&\le \int_{y\in \T^3}\int_{|u^*|\le 3N} w^2_{q,\vartheta,\beta}(u^*)|f(s',y,u^*)|^2\textbf{1}_{\{|f|\le \sqrt{\mu}\}} du^*dy\\
&\indent + \sup_{0\le s'\le T}\|h(s')\|_{L^{\infty}}\int_{y\in \T^3}\int_{|u^*|\le 3N} w_{q,\vartheta,\beta}(u^*)|f(s',y,u^*)|\textbf{1}_{\{|f|\ge \sqrt{\mu}\}} du^*dy\\
&\le C_N\int_{y\in \T^3}\int_{|u^*|\le 3N}|f(s',y,u^*)|^2\textbf{1}_{\{|f|\le \sqrt{\mu}\}} du^*dy\\
&\indent + C_N\sup_{0\le s'\le T}\|h(s')\|_{L^{\infty}}\int_{y\in \T^3}\int_{|u^*|\le 3N} \sqrt{\mu(u^*)}|f(s',y,u^*)|\textbf{1}_{\{|f|\ge \sqrt{\mu}\}} du^*dy\\
&\le C_{N,\bar{M}}\mathcal{E}(F_0),
\end{split}
\end{align}
because $w_{q,\vartheta,\beta}(u^*)$ and $\sqrt{\mu(u^*)}$ are bounded for $|u^*| \le 3N$.
By the change of the variables \eqref{change of variable} and \eqref{mathcalE}, it follows that
\begin{align}\label{I_222 estimate}
\begin{split}
I_{222}
&\le (1+2N)^{-\gamma}\int_0^t e^{-\nu(v)(t-s)}\textbf{1}_{\{|v|\le N\}}\\
&\indent \times\sup_{0\le s' \le s-\delta}\|k_w^{\chi}(v,u)k_w^{\chi}(u,u^*)\|_{L^2_{u,u^*}}\left(\int_{|u|\le 2N}\int_{|u^*|\le 3N}|h(s',x'-(s-s')u,u^*)|^2 du^*du\right)^{1/2} ds\\
&\le C_{\epsilon}(1+2N)^{-\gamma}\nu^{-1}(v)\int_0^t \nu(v)e^{-\nu(v)(t-s)}\textbf{1}_{\{|v|\le N\}}\\
&\indent \times \sup_{0\le s' \le s-\delta}(s-s')^{-\frac{3}{2}}\left(\int_{y \in \T^3}\int_{|u^*|\le 3N}|h(s',y,u^*)|^2 du^*dy\right)^{1/2} ds\\
&\le C_{\epsilon,N,\bar{M}}(1+2N)^{-\gamma}(1+N)^{-\gamma}\delta^{-\frac{3}{2}}\mathcal{E}(F_0)^{\frac{1}{2}}.
\end{split}
\end{align}
For $I_{23}$, the Hölder conjugate $r$ of $p'$ satisfies
\begin{align}\label{holder conjugate of p'}
 1\le r \le \frac{5}{4},\indent \|k_w^{\chi}(v,u)\|_{L^{r}(\{|u|\le 2N\})}\le C_{\epsilon}.
\end{align}
Then $I_{23}$ enjoys
\begin{align*}
I_{23}
&= C\int_0^t \int_{|u|\le 2N}\int_0^{s-\delta} e^{-\nu(v)(t-s)}\textbf{1}_{\{|v|\le N\}}k_w^{\chi}(v,u)\nu(u)e^{-\nu(u)(s-s')}\|h(s')\|_{L^{\infty}}\\
&\indent \times\left(\int_{|u^*|\le 3N}(1+|u^*|)^{-2\beta p'+16}|h(u^*)|^{p'}du^*\right)^{1/p'}ds'duds\\
&\le C\int_0^t e^{-\nu(v)(t-s)} \textbf{1}_{\{|v|\le N\}}\\
&\indent \times \sup_{0\le s' \le s-\delta}\left[\|h(s')\|_{L^{\infty}}\int_{|u|\le 2N}k_w^{\chi}(v,u)\left(\int_{|u^*|\le 3N}(1+|u^*|)^{-2\beta p'+16}|h(u^*)|^{p'}du^*\right)^{1/p'} du\right]ds\\
&\le C\int_0^t  e^{-\nu(v)(t-s)} \textbf{1}_{\{|v|\le N\}}\\
&\indent \times\sup_{0\le s' \le s-\delta}\left[\|h(s')\|_{L^{\infty}}\|k_w^{\chi}(v,u)\|_{L^{r}_u}\left(\int_{|u|\le 2N}\int_{|u^*|\le 3N}(1+|u^*|)^{-2\beta p'+16}|h(u^*)|^{p'}du^*du\right)^{1/p'}\right]ds\\
&\le C_{\epsilon}\int_0^t  e^{-\nu(v)(t-s)} \textbf{1}_{\{|v|\le N\}}\sup_{0\le s' \le s-\delta}\left[\|h(s')\|^{2-\frac{1}{p'}}_{L^{\infty}}\left(\int_{|u|\le 2N}\int_{|u^*|\le 3N}(1+|u^*|)^{-2\beta p'+16}|h(u^*)|du^*du\right)^{1/p'}\right]ds\\
&\le C_{\epsilon}(2N)^{\frac{3}{2p'}}\int_0^t  e^{-\nu(v)(t-s)} \textbf{1}_{\{|v|\le N\}}\\
&\indent \times\sup_{0\le s' \le s-\delta}\|h(s')\|^{2-\frac{1}{p'}}_{L^{\infty}}\left\{\|(1+|u^*|)^{-2\beta p'+16}\|_{L^2_{u^*}}\left(\int_{|u|\le 2N}\int_{|u^*|\le 3N}|h(s',x'-(s-s')u,u^*)|^2du^*du\right)^{1/2}\right\}^{1/p'}ds.\\
\end{align*}
By the change of the variables \eqref{change of variable} and \eqref{mathcalE}, it follows that
\begin{align}\label{I_23 estimate}
\begin{split}
I_{23}
&\le C_{\epsilon}(2N)^{\frac{3}{2p'}}\int_0^t  e^{-\nu(v)(t-s)}\textbf{1}_{\{|v|\le N\}}\\
&\indent \times\sup_{0\le s' \le s-\delta}\|h(s')\|^{2-\frac{1}{p'}}_{L^{\infty}}\left\{\|(1+|u^*|)^{-2\beta p'+16}\|_{L^2_{u^*}}\left(\int_{|u|\le 2N}\int_{|u^*|\le 3N}|h(s',x'-(s-s')u,u^*)|^2du^*du\right)^{1/2}\right\}^{1/p'}ds\\
&\le C_{\epsilon}N^{\frac{3}{2p'}}\nu^{-1}(v)\int_0^t  \nu(v)e^{-\nu(v)(t-s)}\textbf{1}_{\{|v|\le N\}}\\
&\indent \times\sup_{0\le s' \le s-\delta}\left[\|h(s')\|^{2-\frac{1}{p'}}_{L^{\infty}}\left\{(s-s')^{-\frac{3}{2}}\left(\int_{y \in \T^3}\int_{|u^*|\le 3N}|h(s',y,u^*)|^2 du^*dy\right)^{1/2}\right\}^{1/p'}\right]ds\\
&\le C_{\epsilon,N,\bar{M}}N^{\frac{3}{2p'}}(1+N)^{-\gamma}\delta^{-\frac{3}{2p'}}\sup_{0\le s' \le T}\|h(s')\|^{2-\frac{1}{p'}}_{L^{\infty}}\mathcal{E}(F_0)^{\frac{1}{2p'}}.
\end{split}
\end{align}
For an estimate of $I_{24}$, we use Lemma \ref{gamma estimate} to get
\begin{align*}
I_{24}
&= C\int_0^t \int_{|u|\le 2N} \int_0^{s-\delta}e^{-\nu(v)(t-s)}e^{-\nu(u)(s-s')}\textbf{1}_{\{|v|\le N\}}k_w^{\chi}(v,u)\nu(u)\|h(s')\|_{L^{\infty}}\\
&\indent \times\left(\int_{|u^*|\le 3N}w_{q,\vartheta,\beta}(u^*)^{-p'}|h(u^*)|^{p'}du^*\right)^{1/p'}ds'duds\\
&\le C\int_0^t \int_{|u|\le 2N}\int_0^{s-\delta} e^{-\nu(v)(t-s)}\nu(u)e^{-\nu(u)(s-s')}\textbf{1}_{\{|v|\le N\}}k_w^{\chi}(v,u)\|h(s')\|_{L^{\infty}}\\
&\indent \times\left(\int_{|u^*|\le 3N}(1+|u^*|)^{-2\beta p'+16}|h(u^*)|^{p'}du^*\right)^{1/p'}ds'duds.\\
\end{align*}
Therefore we can estimate $I_{24}$ similarly to the estimate of $I_{23}$ : 
\begin{align}\label{I_24 estimate}
I_{24}
&\le C_{N,\bar{M}}N^{\frac{3}{2p'}}(1+N)^{-\gamma}\delta^{-\frac{3}{2p'}}\sup_{0\le s' \le T}\|h(s')\|^{2-\frac{1}{p'}}_{L^{\infty}}\mathcal{E}(F_0)^{\frac{1}{2p'}}.
\end{align}
From Lemma \ref{gamma estimate}, we can get
\begin{align}\label{localgamma}
\begin{split}
|w\Gamma_-(f,g)(t)|
&\le C_{\gamma}\nu(v)\|w_{q,\vartheta,\beta}f(t)\|_{L^{\infty}}\|w_{q,\vartheta,\beta}g(t)\|_{L^{\infty}}\left(\int_{\mathbb{R}^3}\frac{1}{w_{q,\vartheta,\beta}(u)^{p'}}du\right)^{1/p'}\\
&\le C_{\gamma}\nu(v)\|w_{q,\vartheta,\beta}f(t)\|_{L^{\infty}}\|w_{q,\vartheta,\beta}g(t)\|_{L^{\infty}},\\
|w\Gamma_+(f,g)(t)|
&\le C_{\gamma}\nu(v)\|w_{q,\vartheta,\beta}f(t)\|_{L^{\infty}}\|w_{q,\vartheta,\beta}g(t)\|_{L^{\infty}}\left(\int_{\mathbb{R}^3}(1+|u|)^{-2\beta p'+16}du\right)^{1/p'}\\
&\le C_{\gamma}\nu(v)\|w_{q,\vartheta,\beta}f(t)\|_{L^{\infty}}\|w_{q,\vartheta,\beta}g(t)\|_{L^{\infty}},
\end{split}
\end{align} 
because $\frac{1}{w_{q,\vartheta,\beta}(u)^{p'}}$ and $(1+|u|)^{-2\beta p'+16}$ are integrable on $\mathbb{R}^3$.
For an estimate of $I_{25}$, note that $\int_{s-\delta}^s e^{-\nu(u)(s-s')} ds' \le \delta$. By \eqref{gamma estimate} and \eqref{localgamma}, we can obtain
\begin{align}\label{I_25 estimate}
\begin{split}
 I_{25}
&= \int_0^t \int_{|u|\le 2N}\int_{s-\delta}^s e^{-\nu(v)(t-s)}k_w^{\chi}(v,u)e^{-\nu(u)(s-s')}\textbf{1}_{\{|v|\le N, |u^*|\le 3N\}}\\
&\indent \times \left[|K_wh(s')| + |w\Gamma_+(f,f)(s')|+|w\Gamma_-(f,f)(s')|\right] ds'duds\\
&\le C\int_0^t \int_{|u|\le 2N}\int_{s-\delta}^s e^{-\nu(v)(t-s)}k_w(v,u)e^{-\nu(u)(s-s')}\textbf{1}_{\{|v|\le N\}}\left[\|h(s')\|_{L^{\infty}} + 2\|h(s')\|_{L^{\infty}}^2\right] ds'duds\\
&\le C\delta\int_0^t \int_{|u|\le 2N} e^{-\nu(v)(t-s)}k_w(v,u)\textbf{1}_{\{|v|\le N\}}\sup_{s-\delta\le s' \le s}[\|h(s')\|_{L^{\infty}} + \|h(s')\|_{L^{\infty}}^2] duds\\
&\le C\delta\nu^{-1}(v)\int_0^t \nu(v)e^{-\nu(v)(t-s)}\textbf{1}_{\{|v|\le N\}}\sup_{s-\delta \le s' \le s}[\|h(s')\|_{L^{\infty}} + \|h(s')\|_{L^{\infty}}^2] ds\\
&\le C\delta(1+N)^{-\gamma}\sup_{0\le s' \le T}[\|h(s')\|_{L^{\infty}} + \|h(s')\|_{L^{\infty}}^2].
\end{split}
\end{align}
Get together \eqref{I_20 estimate}, \eqref{I_21 estimate}, \eqref{I_221 estimate}, \eqref{I_222 estimate}, \eqref{I_23 estimate}, \eqref{I_24 estimate}, and \eqref{I_25 estimate}. Then we gain
\begin{align}\label{I_2 case4}
I_2\textbf{1}_{\{|v|\le N, |u| \le 2N, |u^*|\le 3N\}}
&\le C_{\bar{M}}\epsilon^{\gamma+3} + Ce^{-\frac{\lambda}{2}(1+t)^{\rho}}\|h_0\|_{L^{\infty}} + C_{\epsilon,N,\bar{M},\delta}(\mathcal{E}(F_0)^{\frac{1}{2}}+\mathcal{E}(F_0)^{\frac{1}{2p'}}) + C_{\bar{M},N}\delta.
\end{align}
Join \eqref{I_2 case1}, \eqref{I_2 case2}, \eqref{I_2 case3}, and \eqref{I_2 case4} altogether. Then it follows that
\begin{align}\label{I_2 estimate}
\begin{split}
I_2 &\le  Ce^{-\frac{\lambda}{2}(1+t)^{\rho}}\|h_0\|_{L^{\infty}}+ C_{\bar{M}}\epsilon^{\gamma+3}+ C_{\bar{M}}\left(\frac{1}{N^{\frac{3+\gamma}{2}}}+N^{\gamma}+e^{-\frac{9}{8}N^2}+\frac{1}{(1+3N)^{3+\gamma}}\right)+\frac{C_{\bar{M},\epsilon}}{N^2}+ C_{\bar{M},N}\delta\\
&\indent + C_{\epsilon,\delta,N,\bar{M}}\left(\mathcal{E}(F_0)^{\frac{1}{2}}+\mathcal{E}(F_0)^{\frac{1}{2p'}}\right).
\end{split}
\end{align}

For $I_3$, denote that $h(u) = h(s,x',u)$ and, by Duhamel's principle and Lemma \ref{gamma estimate}, we split $I_3$ into 
\begin{align*}
I_{3}
&= \int_0^{t} G_v(t,s)w\Gamma_+(f,f)(s) ds\\
&\le C\int_0^{t} G_v(t,s)\nu(v)\|h(s)\|_{L^{\infty}}\left(\int_{\R^3} (1+|u|)^{-2\beta p' + 16}|h(u)|^{p'}\right)^{1/p'} ds\\
&\le C\int_0^{t} G_v(t,s)\nu(v)\|h(s)\|_{L^{\infty}}\\
&\indent \times \left(\int_{\R^3} (1+|u|)^{-2\beta p' + 16}\left|G_u(s,0)h_0+\int_0^{s}G_u(s,s')\left[K_wh(s')+w\Gamma_+(f,f)(s')+w\Gamma_-(f,f)(s')\right]ds'\right|^{p'}du\right)^{1/p'} ds\\
&\le I_{31}+I_{32}+I_{33}+I_{34}+I_{35},
\end{align*}
where
\begin{align*}
I_{31}
&:= C\int_0^{t} G_v(t,s)\nu(v)\|h(s)\|_{L^{\infty}}\left(\int_{\R^3} (1+|u|)^{-2\beta p' + 16}|G_u(s,0)h_0|^{p'}du\right)^{1/p'}ds,\\
I_{32}
&:= C\int_0^{t} G_v(t,s)\nu(v)\|h(s)\|_{L^{\infty}}\left\{\int_{\R^3} (1+|u|)^{-2\beta p' + 16}\left(\int_0^{s-\delta}G_u(s,s')|K_wh(s')|ds'\right)^{p'}du\right\}^{1/p'}ds,\\
I_{33}
&:= C\int_0^{t} G_v(t,s)\nu(v)\|h(s)\|_{L^{\infty}}\left\{\int_{\R^3} (1+|u|)^{-2\beta p' + 16}\left(\int_0^{s-\delta}G_u(s,s')|w\Gamma_+(f,f)(s')|ds'\right)^{p'}du\right\}^{1/p'}ds,\\
I_{34}
&:= C\int_0^{t} G_v(t,s)\nu(v)\|h(s)\|_{L^{\infty}}\left\{\int_{\R^3} (1+|u|)^{-2\beta p' + 16}\left(\int_0^{s-\delta}G_u(s,s')|w\Gamma_-(f,f)(s')|ds'\right)^{p'}du\right\}^{1/p'}ds,\\
I_{35}
&:= C\int_0^{t} G_v(t,s)\nu(v)\|h(s)\|_{L^{\infty}}\\
&\indent \times\left(\int_{\R^3} (1+|u|)^{-2\beta p' + 16}\left|\int_{s-\delta}^{s}G_u(s,s')[K_wh(s')+w\Gamma_+(f,f)(s')+w\Gamma_-(f,f)(s')]ds'\right|^{p'}du\right)^{1/p'}ds.
\end{align*}
For $I_{31}$, by \eqref{nutilde estimate}, $I_{31}$ satisfies
\begin{align}\label{I_31 estimate}
\begin{split}
I_{31}
&= C\int_0^{t} G_v(t,s)\nu(v)\|h(s)\|_{L^{\infty}}\left(\int_{\R^3} (1+|u|)^{-2\beta p' + 16}|G_u(s,0)h_0|^{p'}du\right)^{1/p'}ds\\
&\le C\int_0^{t} G_v(t,s)G_u(s,0)\|h(s)\|_{L^{\infty}}\|h_0\|_{L^{\infty}}\left(\int_{\R^3} (1+|u|)^{-2\beta p' + 16}du\right)^{1/p'}ds\\
&\le C\int_0^{t} e^{-C\int_s^t(1+\tau)^{\rho-1}d\tau-C\int_0^s(1+\tau)^{\rho-1}d\tau}\|h(s)\|_{L^{\infty}}\|h_0\|_{L^{\infty}}ds\\
&\le Cte^{-\lambda(1+t)^{\rho}}\|h_0\|_{L^{\infty}}\int_0^t\|h(s)\|_{L^{\infty}}ds\\
&\le Ce^{-\frac{\lambda}{2}(1+t)^{\rho}}\|h_0\|_{L^{\infty}}\int_0^t\|h(s)\|_{L^{\infty}}ds,\\
\end{split}
\end{align}
in that $te^{-\frac{\lambda}{2}(1+t)^{\rho}}$ is bounded for $t\ge 0$.
For $I_{32}$, we divde $I_{32}$ into 
\begin{align*}
I_{32}
&\le I_{321} + I_{322} + I_{323} + I_{324},
\end{align*}
where
\begin{align*}
I_{321}
&:= C\int_0^{t} G_v(t,s)\nu(v)\|h(s)\|_{L^{\infty}}\\
&\indent \times\left\{\int_{|u|\ge N} (1+|u|)^{-2\beta p' + 16}\left(\int_0^{s-\delta}G_u(s,s')\left|\int_{\R^3}k_w(u,u^*)h(s')du^*\right|\right)^{p'}du\right\}^{1/p'}ds,\\
I_{322}
&:= C\int_0^{t} G_v(t,s)\nu(v)\|h(s)\|_{L^{\infty}}\\
&\indent \times\left\{\int_{|u|\le N} (1+|u|)^{-2\beta p' + 16}\left(\int_0^{s-\delta}G_u(s,s')\left|\int_{|u^*|\ge 2N}k_w(u,u^*)h(s')du^*\right|ds'\right)^{p'}du\right\}^{1/p'}ds,\\
I_{323}
&:=C\int_0^{t} G_v(t,s)\nu(v)\|h(s)\|_{L^{\infty}}\\
&\indent \times\left\{\int_{|u|\le N} (1+|u|)^{-2\beta p' + 16}\left(\int_0^{s-\delta}G_u(s,s')\left|\int_{|u^*|\le 2N}k_w^{1-\chi}(u,u^*)h(s')du^*\right|ds'\right)^{p'}du\right\}^{1/p'}ds,\\
I_{324}
&:=C\int_0^{t} G_v(t,s)\nu(v)\|h(s)\|_{L^{\infty}}\\
&\indent \times\left\{\int_{|u|\le N} (1+|u|)^{-2\beta p' + 16}\left(\int_0^{s-\delta}G_u(s,s')\left|\int_{|u^*|\le 2N}k_w^{\chi}(u,u^*)h(s')du^*\right|ds'\right)^{p'}du\right\}^{1/p'}ds.\\
\end{align*}
For $I_{321}$, notice that $(1+|u|)^{-2\beta p' + 16}$ is integrable on $\{|u| \ge N\}$. By Lemma \ref{Lemma_k.1}, we can get  similar to the estimate of $I_{2}$ for $|v|\ge N$
\begin{align}\label{I_321 estimate}
\begin{split}
I_{321}
&= C\int_0^{t} G_v(t,s)\nu(v)\|h(s)\|_{L^{\infty}}\\
&\indent \times\left\{\int_{|u|\ge N} (1+|u|)^{-2\beta p' + 16}\left(\int_0^{s-\delta}G_u(s,s')\left|\int_{\R^3}(k_w^{1-\chi}+k_w^{\chi})(u,u^*)h(s')\right|ds'\right)^{p'}du\right\}^{1/p'}ds\\
&\le C\int_0^{t} G_v(t,s)\nu(v)\|h(s)\|_{L^{\infty}}\\
&\indent \times\left\{\int_{|u|\ge N} (1+|u|)^{-2\beta p' + 16}\left(\int_0^{s-\delta}e^{-\nu(u)(s-s')}\nu(u)\|h(s')\|_{L^{\infty}}\nu^{-1}(u)\left(\mu(u)^{\frac{1-q}{8}}\epsilon^{\gamma+3}+C_{\epsilon}\langle u \rangle^{\gamma-2}\right)ds'\right)^{p'}du\right\}^{1/p'}ds\\
&\le C\int_0^{t} G_v(t,s)\nu(v)\|h(s)\|_{L^{\infty}}\left\{\int_{|u|\ge N} (1+|u|)^{-2\beta p' + 16}\left(\epsilon^{p'(\gamma+3)}+(\frac{C_{\epsilon}}{N^2})^{p'}\right)\left(\sup_{0\le s' \le s-\delta}\|h(s')\|_{L^{\infty}}\right)^{p'}du\right\}^{1/p'}ds\\
&\le \left(C\epsilon^{\gamma+3}+\frac{C_{\epsilon}}{N^2}\right)\int_0^{t} e^{-\nu(v)(t-s)}\nu(v)\|h(s)\|_{L^{\infty}}\sup_{0\le s' \le s-\delta}\|h(s')\|_{L^{\infty}}ds\\
&\le \left(C\epsilon^{\gamma+3}+\frac{C_{\epsilon}}{N^2}\right)\sup_{0\le s \le T}\|h(s)\|_{L^{\infty}}^2.
\end{split}
\end{align}
For $I_{322}$, from Lemma \ref{Lemma_k.2} and \eqref{case2 k_w,2}, we can obtain similar to the estimate of $I_2$ for $|v|\le N$ and $|u|\ge 2N$
\begin{align}\label{I_322 estimate}
\begin{split}
I_{322}
&= C\int_0^{t} G_v(t,s)\nu(v)\|h(s)\|_{L^{\infty}}\\
&\indent \times\left\{\int_{|u|\le N} (1+|u|)^{-2\beta p' + 16}\left(\int_0^{s-\delta}G_u(s,s')\left|\int_{|u^*|\ge 2N}k_w(u,u^*)h(s')du^*\right|ds'\right)^{p'}du\right\}^{1/p'}ds\\
&\le C\int_0^{t} G_v(t,s)\nu(v)\|h(s)\|_{L^{\infty}}\\
&\indent \times\left\{\int_{|u|\le N} (1+|u|)^{-2\beta p' + 16}\left(\int_0^{s-\delta}e^{-\nu(u)(s-s')}\left|\int_{|u^*|\ge 2N}(k_{w,1}+k_{w,2})(u^*,u)h(s')du^*\right|ds'\right)^{p'}du\right\}^{1/p'}ds\\
&\le C\int_0^{t} e^{-\nu(v)(t-s)}\nu(v)\|h(s)\|_{L^{\infty}}\left\{\int_{|u|\le N} (1+|u|)^{-2\beta p' + 16}\left(\sup_{0 \le s \le s-\delta}(N^{\gamma} + \frac{1}{N^{\frac{3+\gamma}{2}}})\|h(s')\|\right)^{p'}du\right\}^{1/p'}ds\\
&\le C\left(N^{\gamma} + \frac{1}{N^{\frac{3+\gamma}{2}}}\right)\sup_{0 \le s \le T}\|h(s)\|_{L^{\infty}}^2,
\end{split}
\end{align}
because $-2\beta p' + 16 < -3$. For $I_{323}$, we use Lemma \ref{Lemma_k.1} to get
\begin{align}\label{I_323 estimate}
\begin{split}
I_{323}
&= C\int_0^{t} G_v(t,s)\nu(v)\|h(s)\|_{L^{\infty}}\\
&\indent \times\left\{\int_{|u|\le N} (1+|u|)^{-2\beta p' + 16}\left(\int_0^{s-\delta}G_u(s,s')\left|\int_{|u^*|\le 2N}k_w^{1-\chi}(u,u^*)h(s')du^*\right|\right)^{p'}du\right\}^{1/p'}ds\\
&\le C\int_0^{t} e^{-\nu(v)(t-s)}\nu(v)\|h(s)\|_{L^{\infty}}\\
&\indent \times\left\{\int_{|u|\le N} (1+|u|)^{-2\beta p' + 16}\left(\int_0^{s-\delta}e^{-\nu(u)(s-s')}\nu(u)\nu(u)^{-1}\mu^{\frac{1-q}{8}}(u)\epsilon^{\gamma+3}\|h(s')\|_{L^{\infty}}\right)^{p'}du\right\}^{1/p'}ds\\
&\le C\int_0^{t} e^{-\nu(v)(t-s)}\nu(v)\|h(s)\|_{L^{\infty}}\left\{\int_{|u|\le N} (1+|u|)^{-2\beta p' + 16}\epsilon^{p'(\gamma+3)}\left(\sup_{0\le s' \le s-\delta}\|h(s')\|_{L^{\infty}}\right)^{p'}du\right\}^{1/p'}ds\\
&\le C\int_0^{t} e^{-\nu(v)(t-s)}\nu(v)\|h(s)\|_{L^{\infty}}\epsilon^{\gamma+3}\sup_{0\le s' \le s-\delta}\|h(s')\|_{L^{\infty}}ds\\
&\le C\epsilon^{\gamma+3}\sup_{0\le s \le T}\|h(s)\|_{L^{\infty}}^2.
\end{split}
\end{align}
For $I_{324}$, the Hölder conjugate $r$ of $p'$ satisfies \eqref{holder conjugate of p'}. Note that $\|(1+|u|)^{-2\beta p' + 16}\|_{L^2_u}<\infty$ and $I_{324}$ can be bounded by
\begin{align*}
I_{324}
&= C\int_0^{t} G_v(t,s)\nu(v)\|h(s)\|_{L^{\infty}}\\
&\indent \times\left\{\int_{|u|\le N} (1+|u|)^{-2\beta p' + 16}\left(\int_0^{s-\delta}G_u(s,s')\left|\int_{|u^*|\le 2N}k_w^{\chi}(u,u^*)h(s')du^*\right|\right)^{p'}du\right\}^{1/p'}ds\\
&\le C(1+N)^{-\gamma}\int_0^{t} G_v(t,s)\nu(v)\|h(s)\|_{L^{\infty}}\\
&\indent \times\sup_{0\le s' \le s-\delta}\left\{\int_{|u|\le N} (1+|u|)^{-2\beta p' + 16}\left(\int_{|u^*|\le 2N}k_w^{\chi}(u,u^*)|h(s')|du^*\right)^{p'}du\right\}^{1/p'}ds\\
&\le C(1+N)^{-\gamma}\int_0^{t} G_v(t,s)\nu(v)\|h(s)\|_{L^{\infty}}\\
&\indent \times\sup_{0\le s' \le s-\delta}\left(\int_{|u|\le N} (1+|u|)^{-2\beta p' + 16}\|k_w^{\chi}(u,u^*)\|_{L^{r}(\R^3_{u^*})}^{p'}\int_{|u^*|\le 2N}|h(s')|^{p'}du^*\right)^{1/p'}ds\\
&\le C_{\epsilon}(1+N)^{-\gamma}\int_0^{t} G_v(t,s)\nu(v)\|h(s)\|_{L^{\infty}}\\
&\indent \times\sup_{0\le s' \le s-\delta}\left(\|h(s')\|^{p'-1}_{L^{\infty}}\int_{|u|\le N}\int_{|u^*|\le 2N} (1+|u|)^{-2\beta p' + 16}|h(s')|du^*du\right)^{1/p'}ds\\
&\le C_{\epsilon}(1+N)^{-\gamma}\int_0^{t} G_v(t,s)\nu(v)\|h(s)\|_{L^{\infty}}\\
&\indent \times\sup_{0\le s' \le s-\delta}\left\{\|h(s')\|^{p'-1}_{L^{\infty}}\|(1+|u|)^{-2\beta p' + 16}\|_{L^2_{u,u^*}}\left(\int_{|u|\le N}\int_{|u^*|\le 2N}|h(s')|^2du^*du\right)^{1/2}\right\}^{1/p'}ds\\
&= C_{\epsilon}(1+N)^{-\gamma}\int_0^{t} G_v(t,s)\nu(v)\|h(s)\|_{L^{\infty}}\\
&\indent \times\sup_{0\le s' \le s-\delta}\left\{(2N)^{\frac{3}{2}}\|h(s')\|^{p'-1}_{L^{\infty}}\|(1+|u|)^{-2\beta p' + 16}\|_{L^2_{u}}\left(\int_{|u|\le N}\int_{|u^*|\le 2N}|h(s',x'-(s-s')u,u^*)|^2du^*du\right)^{1/2}\right\}^{1/p'}ds.\\
\end{align*}
By the change of the variables \eqref{change of variable} and \eqref{mathcalE}, $I_{324}$ satisfies
\begin{align}\label{I_324 estimate}
\begin{split}
I_{324}
&\le C_{\epsilon}(1+N)^{-\gamma}\int_0^{t} G_v(t,s)\nu(v)\|h(s)\|_{L^{\infty}}\\
&\indent \times\sup_{0\le s' \le s-\delta}\left\{(2N)^{\frac{3}{2}}\|h(s')\|^{p'-1}_{L^{\infty}}\|(1+|u|)^{-2\beta p' + 16}\|_{L^2_{u}}\left(\int_{|u|\le N}\int_{|u^*|\le 2N}|h(s',x'-(s-s')u,u^*)|^2du^*du\right)^{1/2}\right\}^{1/p'}ds\\
&\le C_{\epsilon}(1+N)^{-\gamma}\int_0^{t} G_v(t,s)\nu(v)\|h(s)\|_{L^{\infty}}\\
&\indent \times\sup_{0\le s' \le s-\delta}\left\{(2N)^{\frac{3}{2}}\|h(s')\|^{p'-1}_{L^{\infty}}\|(1+|u|)^{-2\beta p' + 16}\|_{L^2_{u}}(s-s')^{-\frac{3}{2}}\left(\int_{y\in\T^3}\int_{|u^*|\le 2N}|h(s',y,u^*)|^2du^*dy\right)^{1/2}\right\}^{1/p'}ds\\
&\le C_{\epsilon,N,\bar{M}}(1+N)^{-\gamma}N^{\frac{3}{2p'}}\delta^{-\frac{3}{2p'}}\sup_{0\le s \le T}\|h(s)\|_{L^{\infty}}^{2-\frac{1}{p'}}\mathcal{E}(F_0)^{\frac{1}{2p'}}.
\end{split}
\end{align}
Combining \eqref{I_321 estimate}, \eqref{I_322 estimate}, \eqref{I_323 estimate}, and \eqref{I_324 estimate}, we can get
\begin{align}\label{I_32 estimate}
I_{32} \le C_{\bar{M}}\epsilon^{\gamma+3} + C_{\bar{M}}\left(N^{\gamma} + \frac{1}{N^{\frac{\gamma+3}{2}}}\right) + \frac{C_{\bar{M},\epsilon}}{N^2} + C_{\bar{M},\epsilon, N, \delta}\mathcal{E}(F_0)^{\frac{1}{2p'}}.
\end{align}
For $I_{33}$, it follows from Lemma \ref{gamma estimate} that
\begin{align*}
I_{33}
&= C\int_0^{t} G_v(t,s)\nu(v)\|h(s)\|_{L^{\infty}}\left\{\int_{\R^3} (1+|u|)^{-2\beta p' + 16}\left(\int_0^{s-\delta}G_u(s,s')|w\Gamma_+(f,f)(s')|ds'\right)^{p'}du\right\}^{1/p'}ds\\
&\le C\int_0^{t} G_v(t,s)\nu(v)\|h(s)\|_{L^{\infty}}\\
&\indent \times\left\{\int_{\R^3} (1+|u|)^{-2\beta p' + 16}\left(\int_0^{s-\delta}e^{-\nu(u)(s-s')}\nu(u)\|h(s')\|_{L^{\infty}}\left[\int_{\R^3}(1+|u^*|)^{-2\beta p' + 16}|h(u^*)|^{p'}du^*\right]^{1/p'}ds\right)^{p'}du\right\}^{1/p'}ds\\
&\le C\int_0^{t} G_v(t,s)\nu(v)\|h(s)\|_{L^{\infty}}\\
&\indent \times\left(\sup_{0\le s' \le s-\delta}\|h(s')\|_{L^{\infty}}^{p'}\int_{\R^3} (1+|u|)^{-2\beta p' + 16}\int_{\R^3}(1+|u^*|)^{-2\beta p' + 16}|h(u^*)|^{p'}du^*dsdu\right)^{1/p'}ds\\
&\le I_{331} + I_{332},
\end{align*}
where 
\begin{align*}
I_{331}
&:= C\int_0^{t} G_v(t,s)\nu(v)\|h(s)\|_{L^{\infty}}\\
&\indent \times\left(\sup_{0\le s' \le s-\delta}\|h(s')\|_{L^{\infty}}^{p'}\int_{\R^3} (1+|u|)^{-2\beta p' + 16}\int_{|u^*|\ge 3N}(1+|u^*|)^{-2\beta p' + 16}|h(u^*)|^{p'}du^*dsdu\right)^{1/p'}ds,\\
I_{332}
&:= C\int_0^{t} G_v(t,s)\nu(v)\|h(s)\|_{L^{\infty}}\\
&\indent \times\left(\sup_{0\le s' \le s-\delta}\|h(s')\|_{L^{\infty}}^{p'}\int_{\R^3} (1+|u|)^{-2\beta p' + 16}\int_{|u^*|\le 3N}(1+|u^*|)^{-2\beta p' + 16}|h(u^*)|^{p'}du^*dsdu\right)^{1/p'}ds.\\
\end{align*}
For $I_{331}$, we can obtain
\begin{align}\label{I_331 estimate}
\begin{split}
I_{331}
&= C\int_0^{t} G_v(t,s)\nu(v)\|h(s)\|_{L^{\infty}}\\
&\indent \times\left(\sup_{0\le s' \le s-\delta}\|h(s')\|_{L^{\infty}}^{p'}\int_{\R^3} (1+|u|)^{-2\beta p' + 16}\int_{|u^*|\ge 3N}(1+|u^*|)^{-2\beta p' + 16}|h(u^*)|^{p'}du^*dsdu\right)^{1/p'}ds\\
&\le C\int_0^{t} G_v(t,s)\nu(v)\|h(s)\|_{L^{\infty}}\\
&\indent \times\left(\sup_{0\le s' \le s-\delta}\|h(s')\|_{L^{\infty}}^{2p'}\int_{\R^3} (1+|u|)^{-2\beta p' + 16}\int_{|u^*|\ge 3N}(1+|u^*|)^{-2\beta p' + 16}du^*dsdu\right)^{1/p'}ds\\
&\le \frac{C}{1+3N}\int_0^{t} G_v(t,s)\nu(v)\|h(s)\|_{L^{\infty}}\\
&\indent \times\left(\sup_{0\le s' \le s-\delta}\|h(s')\|_{L^{\infty}}^{2p'}\int_{\R^3} (1+|u|)^{-2\beta p' + 16}\int_{|u^*|\ge 3N}(1+|u^*|)^{-p'(2\beta-1) + 16}du^*dsdu\right)^{1/p'}ds\\
&\le \frac{C}{1+3N}\int_0^{t} e^{-\nu(v)(t-s)}\nu(v)\|h(s)\|_{L^{\infty}}\sup_{0\le s' \le s-\delta}\|h(s')\|_{L^{\infty}}^{2} ds\\
&\le \frac{C}{1+3N}\sup_{0\le s' \le T}\|h(s')\|_{L^{\infty}}^{3}, 
\end{split}
\end{align}
since $(-p')(2\beta -1)+16 < -3$.
By the change of the variables \eqref{change of variable} and \eqref{mathcalE}, $I_{332}$ enjoys
\begin{align}\label{I_332 estimate}
\begin{split}
I_{332}
&= C\int_0^{t} G_v(t,s)\nu(v)\|h(s)\|_{L^{\infty}}\\
&\indent \times\left(\sup_{0\le s' \le s-\delta}\|h(s')\|_{L^{\infty}}^{p'}\int_{\R^3} (1+|u|)^{-2\beta p' + 16}\int_{|u^*|\le 3N}(1+|u^*|)^{-2\beta p' + 16}|h(u^*)|^{p'}du^*dsdu\right)^{1/p'}ds\\
&\le C\int_0^{t} G_v(t,s)\nu(v)\|h(s)\|_{L^{\infty}}\sup_{0\le s' \le s-\delta}\|h(s')\|_{L^{\infty}}^{2-\frac{1}{p'}}\\
&\indent \times\left\{\sup_{0\le s' \le s-\delta}\|\{(1+|u|)(1+|u^*|)\}^{-2\beta p' + 16}\|_{L^2}\left(\int_{\R^3}\int_{|u^*|\le 3N}|h(s',x'-(s-s')u,u^*)|^2du^*du\right)^{1/2}\right\}^{1/p'}ds\\
&\le C\int_0^{t} G_v(t,s)\nu(v)\|h(s)\|_{L^{\infty}}\sup_{0\le s' \le s-\delta}\|h(s')\|_{L^{\infty}}^{2-\frac{1}{p'}}\\
&\indent \times\left\{\sup_{0\le s' \le s-\delta}\|\{(1+|u|)(1+|u^*|)\}^{-2\beta p' + 16}\|_{L^2}\left(\int_{\R^3}\int_{|u^*|\le 3N}|h(s',x'-(s-s')u,u^*)|^2du^*du\right)^{1/2}\right\}^{1/p'}ds\\
&\le C\int_0^{t} G_v(t,s)\nu(v)\|h(s)\|_{L^{\infty}}\sup_{0\le s' \le s-\delta}\|h(s')\|_{L^{\infty}}^{2-\frac{1}{p'}}\\
&\indent \times\left\{\sup_{0\le s' \le s-\delta}\|\{(1+|u|)(1+|u^*|)\}^{-2\beta p' + 16}\|_{L^2}(s-s')^{-\frac{3}{2}}\left(\int_{y\in\T^3}\int_{|u^*|\le 3N}|h(s',y,u^*)|^2du^*dy\right)^{1/2}\right\}^{1/p'}ds\\
&\le C_{N,\bar{M}}\int_0^{t} G_v(t,s)\nu(v)\|h(s)\|_{L^{\infty}}\sup_{0\le s' \le s-\delta}\|h(s')\|_{L^{\infty}}^{2-\frac{1}{p'}}\delta^{-\frac{3}{2p'}}\mathcal{E}(F_0)^{\frac{1}{2p'}}ds\\
&\le C_{N,\bar{M}}\delta^{-\frac{3}{2p'}}\sup_{0\le s \le T}\|h(s)\|_{L^{\infty}}^{3-\frac{1}{p'}}\mathcal{E}(F_0)^{\frac{1}{2p'}}.
\end{split}
\end{align}
Collecting \eqref{I_331 estimate} and \eqref{I_332 estimate}, we can get
\begin{align}\label{I_33 estimate}
I_{33} \le \frac{C_{\bar{M}}}{1+3N} + C_{\bar{M}, \delta, N}\mathcal{E}(F_0)^{\frac{1}{2p'}}.
\end{align}
For $I_{34}$, it follows that
\begin{align}
\begin{split}
I_{34}
&=C\int_0^{t} G_v(t,s)\nu(v)\|h(s)\|_{L^{\infty}}\left\{\int_{\R^3} (1+|u|)^{-2\beta p' + 16}\left(\int_0^{s-\delta}G_u(s,s')|w\Gamma_-(f,f)(s')|ds'\right)^{p'}du\right\}^{1/p'}ds\\
&\le C\int_0^{t} G_v(t,s)\nu(v)\|h(s)\|_{L^{\infty}}\\
&\indent \times\left\{\int_{\R^3} (1+|u|)^{-2\beta p' + 16}\left(\int_0^{s-\delta}G_u(s,s')\nu(u)\|h(s')\|_{L^{\infty}}\left[\int_{\R^3}\frac{1}{w_{q,\vartheta,\beta}(u^*)^{p'}}|h(u^*)|^{p'}du^*\right]^{1/p'}ds\right)^{p'}du\right\}^{1/p'}ds\\
&\le C\int_0^{t} G_v(t,s)\nu(v)\|h(s)\|_{L^{\infty}}\\
&\indent \times\left\{\int_{\R^3} (1+|u|)^{-2\beta p' + 16}\left(\int_0^{s-\delta}G_u(s,s')\nu(u)\|h(s')\|_{L^{\infty}}\left[\int_{\R^3}(1+|u^*|)^{-2\beta p' + 16}|h(u^*)|^{p'}du^*\right]^{1/p'}ds\right)^{p'}du\right\}^{1/p'}ds.\\
\end{split}
\end{align}
Therefore we can gain similar to the estimate of $I_{33}$ 
\begin{align}\label{I_34 estimate}
I_{34} \le  \frac{C_{\bar{M}}}{1+3N} + C_{N,\bar{M},\delta}\mathcal{E}(F_0)^{\frac{1}{2p'}}.
\end{align} 
For $I_{35}$, notice that $\int_{s-\delta}^s G_u(s,s') ds' \le \int_{s-\delta}^s ds' = \delta$ and it follows from \eqref{localgamma} that
\begin{align}\label{I_35 estimate}
\begin{split}
I_{35}
&= C\int_0^{t} G_v(t,s)\nu(v)\|h(s)\|_{L^{\infty}}\\
&\indent \times\left(\int_{\R^3} (1+|u|)^{-2\beta p' + 16}\left|\int_{s-\delta}^{s}G_u(s,s')[K_wh(s')+w\Gamma_+(f,f)(s')+w\Gamma_-(f,f)(s')]ds'\right|^{p'}du\right)^{1/p'}ds\\
&\le C\int_0^{t} G_v(t,s)\nu(v)\|h(s)\|_{L^{\infty}}\left\{\sup_{s-\delta\le s' \le s}\delta\left(\|h(s')\|_{L^{\infty}}^{p'}+\|h(s')\|_{L^{\infty}}^{2p'}\right)\right\}^{1/p'}ds\\
&\le C\delta^{\frac{1}{p'}}[\sup_{0 \le s \le T}\|h(s)\|_{L^{\infty}}^2+\sup_{0 \le s \le T}\|h(s)\|_{L^{\infty}}^3].
\end{split}
\end{align}
Get together \eqref{I_31 estimate}, \eqref{I_32 estimate}, \eqref{I_33 estimate}, \eqref{I_34 estimate}, and \eqref{I_35 estimate}. Then we obtain
\begin{align}\label{I_3 estimate}
\begin{split}
I_3 &\le Ce^{-\frac{\lambda}{2}(1+t)^{\rho}}\|h_0\|_{L^{\infty}}\int_0^t\|h(s)\|_{L^{\infty}}ds + C_{\bar{M}}\epsilon^{\gamma+3}+ C_{\bar{M}}\delta^{\frac{1}{p'}} + C_{\epsilon,\delta,N,\bar{M}}\mathcal{E}(F_0)^{\frac{1}{2p'}}\\
&\indent + C_{\bar{M}}\left(\frac{1}{N^{\frac{3+\gamma}{2}}}+N^{\gamma}+\frac{1}{1+N}\right)+\frac{C_{\epsilon,\bar{M}}}{N^2}.
\end{split}
\end{align}
For $I_{4}$, we bound $I_{4}$ from Lemma \ref{gamma estimate} by
\begin{align*}
I_{4}
&= \int_0^{t} G_v(t,s)w\Gamma_-(f,f)(s) ds\\
&\le C\int_0^{t} G_v(t,s)\nu(v)\|h(s)\|_{L^{\infty}}\left(\int_{\R^3} \frac{1}{w_{q,\vartheta,\beta}(u)^{p'}}|h(u)|^{p'}\right)^{1/p'} ds\\
&\le C\int_0^{t} G_v(t,s)\nu(v)\|h(s)\|_{L^{\infty}}\left(\int_{\R^3} (1+|u|)^{-2\beta p' + 16}|h(u)|^{p'}\right)^{1/p'} ds.
\end{align*}
Therefore we can get similar to the estimate of $I_{3}$
\begin{align}\label{I_4 estimate}
\begin{split}
I_4 &\le Ce^{-\frac{\lambda}{2}(1+t)^{\rho}}\|h_0\|_{L^{\infty}}\int_0^t\|h(s)\|_{L^{\infty}}ds + C_{\bar{M}}\epsilon^{\gamma+3}+ C_{\bar{M}}\delta^{\frac{1}{p'}} + C_{\epsilon,\delta,N,\bar{M}}\mathcal{E}(F_0)^{\frac{1}{2p'}}\\
&\indent + C_{\bar{M}}\left(\frac{1}{N^{\frac{3+\gamma}{2}}}+N^{\gamma}+\frac{1}{1+N}\right)+\frac{C_{\bar{M},\epsilon}}{N^2}.
\end{split}
\end{align}
Notice that for $0<\delta<1$ and sufficiently large $N\gg 1$,
\begin{align*}
\delta \le \delta^{\frac{1}{p'}}, \indent e^{-\frac{9}{8}N^2} \le \frac{1}{N^2} \le \frac{1}{N+1},\indent \frac{1}{(1+N)^{3+\gamma}} \le \frac{1}{N^{\frac{3+\gamma}{2}}}.
\end{align*}
Combining \eqref{I_1 estimate}, \eqref{I_2 estimate}, \eqref{I_3 estimate}, and \eqref{I_4 estimate}, we can get
\begin{align*}
\|h(t)\|_{L^{\infty}} 
&\le Ce^{-\frac{\lambda}{2}(1+t)^{\rho}}\|h_0\|_{L^{\infty}}\left(\int_0^t\|h(s)\|_{L^{\infty}}ds + 1\right) + C_{\bar{M}}\epsilon^{\gamma+3} + C_{\bar{M},N}(\delta+\delta^{\frac{1}{p'}}) + C_{\bar{M},\epsilon,\delta,N}(\mathcal{E}(F_0)^{\frac{1}{2}} +\mathcal{E}(F_0)^{\frac{1}{2p'}} )\\
&\indent + C_{\bar{M}}\left(\frac{1}{N^{\frac{3+\gamma}{2}}}+N^{\gamma}+e^{-\frac{9}{8}N^2}+\frac{1}{(1+N)^{3+\gamma}}+\frac{1}{N+1}\right)+\frac{C_{\bar{M},\epsilon}}{N^2}\\
&\le Ce^{-\frac{\lambda}{2}(1+t)^{\rho}}\|h_0\|_{L^{\infty}}\left(\int_0^t\|h(s)\|_{L^{\infty}}ds + 1\right) + C_{\bar{M}}\epsilon^{\gamma+3} + C_{\bar{M},N}\delta^{\frac{1}{p'}} + C_{\bar{M},\epsilon,\delta,N}(\mathcal{E}(F_0)^{\frac{1}{2}} +\mathcal{E}(F_0)^{\frac{1}{2p'}} )\\
&\indent + C_{\bar{M}}\left(\frac{1}{N^{\frac{3+\gamma}{2}}}+N^{\gamma}+\frac{1}{N+1}\right)+\frac{C_{\bar{M},\epsilon}}{N^2}.\\
\end{align*}
\end{proof}

\section{Proof of the main theorem}
We need to recall the theorem in \cite{SY17} because global existence for the solution in the large amplitude Boltzmann equation depends on the theorem in the small amplitude Boltzmann equation.  
\begin{proposition}\label{small amplitude} \cite{SY17}
Let $0<\vartheta<-\frac{2}{\gamma}$, $-3<\gamma<0$, and $\rho-1 = \frac{(1+\vartheta)\gamma}{2-\gamma}$. Assume that the phase space is $\T_{x}^3\times \R_{v}^3$ and $f_0$ satisfies \eqref{conservation}. If $F_0(t,x,v) = \mu(v) + \sqrt{\mu}(v)f_0(t,x,v)$ and $\|w_{q,\vartheta,\beta}f_0\|_{L^{\infty}}\le \epsilon_1$ sufficiently small, there exists a unique solution $F(t,x,v) = \mu(v) + \sqrt{\mu}(v)f(t,x,v) \ge 0$ to the Boltzmann equation \eqref{def.be} on $\T_{x}^3\times \R_{v}^3$ and f satisfies
\begin{align*}
\|w_{q,\vartheta,\beta}f(t)\|_{L^{\infty}} \le Ce^{-\lambda_1t^{\rho}}\|w_{q,\vartheta,\beta}f_0\|,
\end{align*}
for some $\lambda_1 >0$. 
\end{proposition}
\begin{remark}
	In Proposition \ref{small amplitude}, unlike \cite[Theorem 1.2, page 469]{SY17}, the domain is $\mathbb{T}^3$ and the polynomial term $(1+\abs{v}^2)^{\beta}$ is added to the weight function $w_{q,\vartheta,\beta}$. Thus, it is not exactly the same as Theorem 1.2 in \cite{SY17}. Although there are some differences as above, we could get the same result. Let us introduce the process of obtaining $L^\infty$ decay property. Firstly, if we ignore the boundary effects in \cite{SY17}, we get $\Vert Pf \Vert_{\nu}^2 \lesssim \Vert (I-P)f \Vert_{\nu}^2$ in $\T^3$, which implies that 
\begin{equation*}
	\frac{d}{dt} \Vert f(t) \Vert_{L^2_{x,v}}^2 + \Vert f(t) \Vert_{\nu}^2 \leq 0.
\end{equation*}
Moreover, by using the similar argument of the proof in \cite[Lemma 4.3, Eq. (4.13), page 524]{SY17}, we obtain the following $L^2$ estimate
\begin{equation*}
	\Vert f(t) \Vert_{L^2_{x,v}}^2 +e^{-\lambda_1 t^{\rho}} \int_0^t e^{\lambda_1 s^{\rho}}\Vert f(s)\Vert_{\nu}^2 ds \lesssim e^{-\lambda_1t^\rho} \Vert w_{q/4,\theta,0} f_0 \Vert_{L^2_{x,v}}^2.
\end{equation*}
Secondly, we consider the $L^2\mbox{-}L^\infty$ bootstrap argument. Since we already developed the $K_w$ estimate containing the new weight function $w_{q,\vartheta,\beta}$ in Lemma \ref{Lemma_k.1}, we can derive 
\begin{equation*}
	\Vert w_{q,\vartheta,\beta}f(t) \Vert_{L^\infty_{x,v}} \lesssim e^{-\lambda_1 t^{\rho}} \Vert w_{q,\vartheta,\beta} f_0 \Vert_{L^\infty_{x,v}} + \int_0^t \Vert f(s)\Vert_{L^2_{x,v}} ds, 
\end{equation*}
from similar arguments in the proof of Lemma \ref{prop_apriori}. Here, $f$ is the solution to the linearized Boltzmann equation 
\begin{equation*}
	\partial_t f +v\cdot \nabla_x f +\nu(v)f = Kf \quad \textrm{for } (t,x,v) \in \R_+ \times \T^3 \times \R^3. 
\end{equation*}
Combining $L^2$ estimate and $L^2\mbox{-}L^\infty$ estimate yields 
\begin{equation*}
	\Vert w_{q,\vartheta,\beta} f(t) \Vert_{L^\infty_{x,v}} \lesssim e^{-\lambda_1 t^\rho} \Vert w_{q,\vartheta,\beta} f_0\Vert_{L^\infty_{x,v}}. 
\end{equation*}
For the nonlinear Boltzmann equation, we could derive from Lemma \ref{gamma estimate} 
\begin{equation*}
	\vert w_{q,\vartheta,\beta}(v) \Gamma(f,f)(v) \vert \lesssim \nu(v) \Vert w_{q,\vartheta,\beta} f(t) \Vert_{L^\infty_{x,v}}^2. 
\end{equation*}
Finally, the $\Gamma$ estimate above and the sequential argument of the proof in \cite[page 537]{SY17} gives the global existence and time sub-exponential decay of the nonlinear Boltzmann equation. 
\end{remark}
Now we can prove Theorem \ref{mainthm} applying Proposition \ref{prop_apriori}.
\begin{proof} [Proof of main theorem]
Recall that
\begin{align*}
\sup_{0\le s \le T}\|h(s)\|_{L^{\infty}} \le \bar{M}.
\end{align*}
By Proposition \ref{prop_apriori}, it holds that
\begin{align*}
\|h(t)\|_{L^{\infty}} \le CM_0e^{-\frac{\lambda}{2}(1+t)^{\rho}}\left(1+\int_0^t\|h(s)\|_{L^{\infty}}\right) + D,
\end{align*}
where
\begin{align*}
D := C_{\bar{M}}\epsilon^{\gamma+3} +C_{\bar{M}}\left(N^{\gamma}+\frac{1}{N^{\frac{3+\gamma}{2}}}+\frac{1}{N+1}\right)+\frac{C_{\bar{M},\epsilon}}{N^2} +C_{\bar{M},N}\delta^{\frac{1}{p'}} + C_{\bar{M},N,\epsilon,\delta}\left(\mathcal{E}(F_0)^{\frac{1}{2}}+\mathcal{E}(F_0)^{\frac{1}{2p'}}\right).
\end{align*}
Define
\begin{align*}
G(t) := 1+\int_0^t\|h(s)\|_{L^{\infty}}ds.
\end{align*}
Then we can rewrite
\begin{align}\label{main estimate}
G'(t)-CM_0e^{-\frac{\lambda}{2}(1+t)^{\rho}}G(t) \le D.
\end{align}
Note that $(1+t)^{1-\rho}e^{-\frac{\lambda}{4}(1+t)^{\rho}}$ is bounded and $G(t)>0$ for $t\ge 0$. Then for all $0<t\le T$, we can obtain
\begin{align}\label{Grownall}
\begin{split}
G'(t)-CM_0e^{-\frac{\lambda}{2}(1+t)^{\rho}}G(t)
&=G'(t)-CM_0e^{-\frac{\lambda}{2}(1+t)^{\rho}}G(t)(1+t)^{\rho-1}(1+t)^{1-\rho}\\
&\ge G'(t)-CM_0e^{-\frac{\lambda}{4}(1+t)^{\rho}}G(t)(1+t)^{\rho-1}\\
&= \frac{d}{dt}\left(G(t)\exp\left\{-\frac{4CM_0}{\lambda\rho}\left(1-e^{-\frac{\lambda}{4}(1+t)^{\rho}}\right)\right\}\right).
\end{split}
\end{align}
\eqref{main estimate} and \eqref{Grownall} yield
\begin{align*}
\int_0^t \frac{d}{ds}\left( G(s)\exp\left\{-\frac{4CM_0}{\lambda\rho}\left(1-e^{-\frac{\lambda}{4}(1+s)^{\rho}}\right)\right\}\right) ds \le \int_0^t D ds.
\end{align*}
Then it follows that
\begin{align*}
G(t)\exp\left\{-\frac{4CM_0}{\lambda\rho}\left(1-e^{-\frac{\lambda}{4}(1+t)^{\rho}}\right)\right\}
\le Dt + G(0)\exp\left\{-\frac{4CM_0}{\lambda\rho}\left(1-e^{-\frac{\lambda}{4}}\right)\right\} \le 1+Dt.
\end{align*}
$G(t)$ is bounded for all $0<t\le T$ by
\begin{align}
G(t) \le (1+Dt)\exp\left\{\frac{4CM_0}{\lambda\rho}\left(1-e^{\frac{\lambda}{4}(1+t)^{\rho}}\right)\right\} \le (1+Dt)\exp\left\{\frac{4CM_0}{\lambda\rho}\right \}.
\end{align}
To make D sufficiently small, we can choose sufficiently small $\epsilon>0$ depending on $\bar{M}$, sufficiently large N depending on $\bar{M}$ and $\epsilon$, and sufficiently small $\delta>0$ depending on $\bar{M}$, $\epsilon$, and N. We will determine $\bar{M}$ to depend only on $M_0$. $\epsilon$, $\delta$, and N can be chosen depending only on $M_0$. We can take sufficiently small $\epsilon_0 \in (0,1)$ to make 
\begin{align*}
D\le \min{\left\{\frac{1}{4}\bar{M},\frac{1}{4}\epsilon_1,1\right\}},
\end{align*}
where $\epsilon_1$ was introduced in Proposition \ref{small amplitude}.
Note $1+Dt \le 1+t$ and $(1+t)e^{-\frac{\lambda}{4}(1+t)^{\rho}}$ is bounded for all $t\ge0$. Then we can get for all $0<t\le T$
\begin{align}\label{main result_2}
\begin{split}
\|h(t)\|_{L^{\infty}} 
&\le CM_0e^{-\frac{\lambda}{2}(1+t)^{\rho}}(1+Dt)\exp\left\{\frac{4CM_0}{\lambda\rho}\right\} + D\\
&\le CM_0e^{-\frac{\lambda}{4}(1+t)^{\rho}}\exp\left\{\frac{4CM_0}{\lambda\rho}\right\} + D\\
&\le \frac{1}{4}\bar{M}e^{-\frac{\lambda}{4}(1+t)^{\rho}} + D\\
&\le \frac{1}{4}\bar{M}e^{-\frac{\lambda}{4}t^{\rho}} + D,
\end{split}
\end{align}
where $\bar{M}$ is defined as
\begin{align*}
\bar{M} := 4CM_0\exp\left\{\frac{4CM_0}{\lambda\rho}\right\} + 4M_0.
\end{align*}
Then it follows from \eqref{main result_2} that for all $0< t \le T$,
\begin{align*}
\|h(t)\|_{L^{\infty}} \le \frac{1}{4}\bar{M} + \frac{1}{4}\bar{M} = \frac{1}{2}\bar{M}.
\end{align*}
Therefore, we have proven that if a priori assumption holds,
\begin{align}\label{1/2 priori estimate}
\sup_{0\le t \le T}\|h(t)\|_{L^{\infty}} \le \frac{1}{2}\bar{M}.
\end{align}
Next we should extend local existence of solution to global existence and check that
\begin{align*}
\|h(T)\|_{L^{\infty}} \le \epsilon_1.
\end{align*} 
By Lemma \ref{local estimate}, there exists a time $\hat{t}_0>0$ such that the solution $f(t,x,v)$ of the Boltzmann equation exists for $t \in [0,\hat{t}_0]$ and satisfies
\begin{align*}
\sup_{0\le t \le \hat{t}_0}\|w_{q,\vartheta,\beta}f(t)\|_{L^{\infty}} \le 2\|w_{q,\vartheta,\beta}f_0\|\le \frac{1}{2}\bar{M}.
\end{align*}
Considering $\hat{t}_0$ as the initial time, by Lemma \ref{local estimate}, for some $\tilde{t}>0$, it holds that
\begin{align*}
\sup_{\hat{t}_0\le t \le \hat{t}_0+\tilde{t}}\|w_{q,\vartheta,\beta}f(t)\|_{L^{\infty}}\le 2\|w_{q,\vartheta,\beta}f(\hat{t}_0)\|_{L^{\infty}} \le \bar{M},
\end{align*}
implying, by Lemma \ref{1/2 priori estimate},
\begin{align*}
\sup_{0 \le t \le \hat{t}_0+\tilde{t}}\|w_{q,\vartheta,\beta}f(t)\|_{L^{\infty}} \le \frac{1}{2}\bar{M}.
\end{align*}
Define 
\begin{align*}
T:=\left(\frac{4}{\lambda}\left[\ln \bar{M} + |\ln \epsilon_1|\right]\right)^{\frac{1}{\rho}},
\end{align*}
where $\epsilon_1$ was introduced in Proposition \ref{small amplitude}. We can extend the local existence of the solution to $0\le t\le T$ as above and we can use \eqref{main result_2} to gain
\begin{align*}
\|w_{q,\vartheta,\beta}f(T)\|_{L^{\infty}} 
&\le \frac{1}{4}\bar{M}e^{-\frac{\lambda}{4}T^{\rho}} + D \\
&\le \frac{1}{4}\epsilon_1 + \frac{1}{4}\epsilon_1 < \epsilon_1.
\end{align*}
Therefore we show the global existence and uniqueness of the solution to the Boltzmann equation by Proposition \ref{small amplitude}. For all $t\ge T$, we obtain from Proposition \ref{small amplitude}
\begin{align*}
\|h(t)\| \le C\|h(T)\|_{L^{\infty}}e^{-\lambda(t-T)^{\rho}} \le C\epsilon_1e^{-\lambda t^{\rho}}.
\end{align*}
Taking $\lambda_0 := \min{\left\{\frac{\lambda}{4},\lambda_1\right\}}$, we have that
\begin{align*}
\|w_{q,\vartheta,\beta}f(t)\|_{L^{\infty}} \le C\bar{M}e^{-\lambda_0t^{\rho}} \le \left(4CM_0\exp\left\{\frac{4CM_0}{\lambda_0\rho}\right\} + 4M_0\right)e^{-\lambda_0 t^{\rho}},
\end{align*}
for all $t\ge 0$. 
\end{proof}

\appendix

\section{Local Existence and Uniqueness}

\begin{lemma}\label{local estimate}
	Let $0<q<1$ and $0\le\vartheta<-\frac{2}{\gamma}$ be fixed in the weight function \eqref{weight}. If $F_0(x,v) = \mu(v) +\sqrt{\mu(v)}f_0(x,v)\ge 0$ and $\|w_{q,\vartheta,\beta}f_0\|<\infty$, then there exists a time $\hat{t}_0>0$ such that the inital value problem \eqref{def.be} and \eqref{id} has a unique non-negative solution $F(t,x,v)=\mu(v) + \sqrt{\mu(v)}f(t,x,v)$ for $t\in[0, \hat{t}_0]$, satisfying 
\begin{equation}\label{locales}
	\sup_{0\le t \le \hat{t}_0} \|w_{q,\vartheta,\beta}f(t)\|_{L^{\infty}} \le 2\|w_{q,\vartheta,\beta}f_0\|_{L^{\infty}}.
\end{equation}
\end{lemma}

\begin{proof}

For the local existence of non-negative solution of the Boltzmann equation \eqref{def.be}, we consider the following iteration : 
\begin{align}\label{localitF}
\begin{split}
\begin{cases}
\p_t F^{n+1} + v\cdot\nabla_xF^{n+1} + F^{n+1}\int_{\mathbb{R}^3}\int_{\mathbb{S}^2}B(v-u,\om)F^n(u) d\om du = Q_+(F^n,F^n)\\
F(0,x,v) = F_0(x,v) \ge 0 , F^0(t,x,v)=\mu(v).
\end{cases}
\end{split}
\end{align}
By induction on n, we can prove that all $F^n$ is non-negative for all $n>0$.\\
Define $I^n(t,x,v) := \int_{\mathbb{R}^3}\int_{\mathbb{S}^2}B(v-u,\om)F^n(u) d\om du$. For $n=1$, by our assumption that $F_0$ is nonneagtive,
\begin{equation*}
F^1(t,x,v) = e^{-\nu(v)t}F_0(x-tv,v) + \int_0^t e^{-\nu(v)(t-s)}\nu(v)\mu(v) ds \ge 0,
\end{equation*}
because $\int_{\R^3}\int_{\mathbb{S}^2}B(u-v,w)\mu(u)dwdu = \nu(u)$. Then, we suppose that $F^n$ is nonnegative for $n=1,2,\cdots, k$.
By Duhamel's principle, we get
\begin{equation*}
F^{k+1} = e^{-\int_0^t I^k(\tau,x,v)d\tau}F_0 + \int_0^t e^{-\int_s^t I^k(\tau,x,v)d\tau}Q_+(F^k,F^k) ds \ge 0,
\end{equation*}
since $F_0 \ge 0$ and $Q_+(F^n,F^n) \ge 0$. Therefore $F^n$ is nonnegative for all $n>0$.\\
Hence we can rewrite the above iteration \eqref{localitF} for $h=w_{q,\vartheta,\beta}f$ as follow : 
\begin{align}\label{localitFh}
\begin{split}
\begin{cases}
 (\p_t + v\cdot\nabla_x + \tilde{\nu})h^{n+1}(t) = K_wh^n(t) + w\Gamma_+(f^n,f^n) - w\Gamma_-(f^n,f^{n+1})\\
h^{n+1}(0,x,v)=h_0(x,v), h^0=0.
\end{cases}
\end{split}
\end{align}
We will show that there exists $\hat{t}_1>0$ such that \eqref{localitFh} has a solution over $[0,\hat{t}_1]$ satisfying
\begin{equation}\label{localitFhes}
\sup_{0\le t \le \hat{t}_1}\|h^n(t)\|_{L^{\infty}} \le 2\|h_0\|_{L^{\infty}},
\end{equation}
for all $n>0$ and $t\in[0,\hat{t}_1]$.
For $n=1$, $h^1(t)=G(t,0)h_0$, implying that \eqref{localitFhes} holds for $n=1$.\\
Then we suppose that \eqref{localitFhes} holds for $n=1,2,\cdots,k$. 
By Duhamel's principle and \eqref{localgamma}, 
\begin{align*}
	|h^{k+1}(t)|
&\le G(t,0)\|h_0\|_{L^{\infty}} + \int_0^t G(t,s)[ |K_wh^k(s)|+|w\Gamma_+(f^k,f^k)(s)|+|w\Gamma_-(f^k,f^{k+1})(s)| ] ds\\
&\le \|h_0\|_{L^{\infty}} + C\hat{t}_1[\sup_{0\le s\le \hat{t}_1}\|h^k(s)\|_{L^{\infty}}+\sup_{0\le s\le \hat{t}_1}\|h^k(s)\|^2_{L^{\infty}}+\sup_{0\le s\le \hat{t}_1}\|h^k(s)\|_{L^{\infty}}\sup_{0\le s\le \hat{t}_1}\|h^{k+1}(s)\|_{L^{\infty}}]\\
&\le \|h_0\|_{L^{\infty}} + C\hat{t}_1\|h_0\|_{L^{\infty}}\left(\|h_0\|_{L^{\infty}}+1\right)+ C\hat{t}_1\|h_0\|_{L^{\infty}}\sup_{0\le s\le \hat{t}_1}\|h^{k+1}(s)\|_{L^{\infty}}.
\end{align*}
Taking $\hat{t}_1\le \min{\left\{\frac{1}{3}\left\{C\left(\|h_0\|_{L^{\infty}}+1\right)\right\}^{-1}, \frac{1}{3}\left(C\|h_0\|_{L^{\infty}}\right)^{-1}\right\}}$, we can get
\begin{align*}
\frac{2}{3}\sup_{0\le s\le \hat{t}_1}\|h^{k+1}(s)\|_{L^{\infty}} \le \frac{4}{3}\|h_0\|_{L^{\infty}}.
\end{align*}
By induction on $n$, \eqref{localitFhes} holds for all $n>0$ and $t\in[0,\hat{t}_1]$. 
For proving the convergence of $\{h^n\}$, we consider $(h^{n+1}-h^{n})$. $(h^{n+1}-h^{n})$ is the solution of the following equation : 
\begin{align*}
\begin{cases}
(\p_t + v\cdot\nabla_x + \tilde{\nu})(h^{n+1}-h^n) &= K_w(h^n-h^{n-1}) + w\Gamma_+(f^n,f^n)\\
&\indent{ }-w\Gamma_+(f^{n-1},f^{n-1}) - w\Gamma_-(f^n,f^{n+1}) + w\Gamma_+(f^{n-1},f^{n})\\
(h^{n+1}-h^n)(0)=0.
\end{cases}
\end{align*}
Hence it holds that
\begin{align}\label{gamma-}
\begin{split}
w\Gamma(f,g)-w\Gamma(h,l) 
&= w\Gamma(f,g)-w\Gamma(h,g)+w\Gamma(h,g)-w\Gamma(h,l)\\
&= w\Gamma(f-h,g) + w\Gamma(h,g-l).
\end{split}
\end{align}
\eqref{gamma-} holds although $w\Gamma$ is replaced to $w\Gamma_-$ or $w\Gamma_+$. Applying Duhamel's principle and \eqref{gamma-}, we have
\begin{align*}
|h^{n+1}(t)-h^n(t)|
&\le \int_0^t G(t,s)\{|K_w(h^n-h^{n-1})(s)| + |w\Gamma_+(f^n-f^{n-1},f^n)(s)| + |w\Gamma_+(f^{n-1},f^n-f^{n-1})(s)|\\
&\indent +|w\Gamma_-(f^n-f^{n-1},f^{n+1})(s)|+ |w\Gamma_-(f^{n-1},f^{n+1}-f^n)(s)|\}ds\\
&\le C\hat{t}_0(1+\sup_{0\le s \le \hat{t}_0}\|h^n(s)\|_{L^{\infty}} + \sup_{0\le s \le \hat{t}_0}\|h^{n-1}(s)\|_{L^{\infty}}+\sup_{0\le s \le \hat{t}_0}\|h^{n+1}(s)\|_{L^{\infty}})\sup_{0\le s \le \hat{t}_0}\|h^n(s)-h^{n-1}(s)\|_{L^{\infty}}\\
&\indent + C\hat{t}_0\sup_{0\le s \le \hat{t}_0}\|h^{n-1}(s)\|_{L^{\infty}}\sup_{0\le s \le \hat{t}_0}\|h^{n+1}(s)-h^n(s)\|_{L^{\infty}}\\
&\le C_1\hat{t}_0\left\{(1+\|h_0\|_{L^{\infty}})\sup_{0\le s \le \hat{t}_0}\|h^n(s)-h^{n-1}(s)\|_{L^{\infty}} + \|h_0\|_{L^{\infty}}\sup_{0\le s \le \hat{t}_0}\|h^{n+1}(s)-h^n(s)\|_{L^{\infty}}\right\},
\end{align*}
where $C_1 = 4C$. Take $C'=\max{\{C_1, C_2\}}$ and $\hat{t}_0\le \min{\{\hat{t}_1, \frac{1}{3}\{C'(\|h_0\|_{L^{\infty}}+1)\}^{-1}, \frac{1}{3}(C'\|h_0\|_{L^{\infty}})^{-1}\}}$, where $C_2$ will be determined later in \eqref{local uniqueness}. Then it follows that
\begin{align*}
 \frac{2}{3}\sup_{0\le s \le \hat{t}_0}\|h^{n+1}(s)-h^n(s)\|_{L^{\infty}} \le \frac{1}{3}\sup_{0\le s \le \hat{t}_0}\|h^{n}(s)-h^{n-1}(s)\|_{L^{\infty}}.
\end{align*}
Therefore $\{h^n\}$ is a convergent sequence and we can denote $h^n \to h$, and $F^n \to F$ as $n\to \infty$. Since all $F^n$ is non-negative, $F$ is non-negative for $t\in[0,\hat{t}_0]$, and \eqref{localitFhes} implies \eqref{locales}.
For the uniqueness of the local solution, suppose that there is another solution $g$ to the Boltzmann equation with the same initial condition as $f$ satisfying
\begin{equation}\label{localges}
\sup_{0\le t \le \hat{t}_0} \|w_{q,\vartheta,\beta}g(t)\|_{L^{\infty}} \le 2\|w_{q,\vartheta,\beta}f_0\|_{L^{\infty}},
\end{equation}
and set $h_1:=w_{q,\vartheta,\beta}g$. Then by \eqref{localges}, we obtain
\begin{align}\label{local uniqueness}
\begin{split}
|h_1(t)-h(t)|
& \le \int_0^t G(t,s)\{|K_{w}(h_1-h)(s)| + |w\Gamma(f-g,f)(s)| + |w\Gamma(g,f-g)(s)|\}\\
& \le C\hat{t}_0\sup_{0\le s\le\hat{t}_0}(\|h_1(s)\|_{L^{\infty}}+\|h(s)\|_{L^{\infty}}+1)\sup_{0\le s\le\hat{t}_0}\|h_1(s)-h(s)\|_{L^{\infty}}\\
&\le  C_2\hat{t}_0(\|h_0\|_{L^{\infty}}+1)\sup_{0\le s\le\hat{t_0}}\|h_1(s)-h(s)\|_{L^{\infty}}\\
&\le \frac{1}{2}\sup_{0\le s\le\hat{t_0}}\|h_1(s)-h(s)\|_{L^{\infty}},
\end{split}
\end{align}
where $C_2 = 4C$, implying that $h_1=h$.\\
\end{proof}

\noindent{\bf Acknowledgments.} DL and GK thank Renjun Duan for fruitful discussion. They also thank Yong Wang for his valuable comment about the paper \cite{DHWZ19}. DL and GK are supported by the National Research Foundation of Korea(NRF) grant funded by the Korea government(MSIT)(No. NRF-2019R1C1C1010915). DL is also supported by the POSCO Science Fellowship of POSCO TJ Park Foundation.

\bibliographystyle{plain}

\end{document}